\documentclass[notitlepage]{scrartcl}
\usepackage[shortlabels]{enumitem}
\usepackage{amsmath}
\usepackage{amssymb}
\usepackage{amsthm}
\usepackage{abstract}
\usepackage{xcolor}
\usepackage{comment}
\usepackage{mathtools}
\usepackage{graphicx}
\usepackage{caption}
\usepackage{subcaption}
\usepackage{float}
\usepackage{enumitem}
\usepackage{hyperref}
\hypersetup{
    colorlinks=true,
    linkcolor=blue,
    filecolor=magenta,      
    urlcolor=cyan
    }
\makeatletter
\newcommand*{\barfix}[2][.175ex]{%
  \mathpalette{\@barfix{#1}}{#2}%
}
\newcommand*{\@barfix}[3]{%
  \vbox{%
    \kern#1\relax
    \hbox{$#2#3\m@th$}%
  }%
}
\makeatother

\newcommand{\sat}{\textit{sat}}
\renewcommand{\exp}{\text{exp}}

\newtheorem{theorem}{Theorem}
\newtheorem{thm}{Theorem}[section]
\newtheorem{observation}[thm]{Observation}
\newtheorem{corollary}[thm]{Corollary}
\newtheorem{lemma}[thm]{Lemma}
\newtheorem{claim}[thm]{Claim}

\newcommand{\footremember}[2]{%
    \footnote{#2}
    \newcounter{#1}
    \setcounter{#1}{\value{footnote}}%
}
\newcommand{\footrecall}[1]{%
    \footnotemark[\value{#1}]%
} 
\newcommand\wsat{\mathop{\mbox{$w$-$\mathit{sat}$}}}

\renewcommand{\Pr}{\mathbb{P}}

\usepackage[margin=1in]{geometry}
\begin{document}

\title{Saturation in Random Hypergraphs}
\author{%
Sahar Diskin \footremember{alley}{School of Mathematical Sciences, Tel Aviv University, Tel Aviv 6997801, Israel.}%
\and Ilay Hoshen \footrecall{alley}%
\and D\'aniel Kor\'andi \footremember{trailer}{Z\"urich, Switzerland.}%
\and Benny Sudakov \footremember{trailer2}{Department of Mathematics, ETH, Z\"urich, Switzerland. Research supported in part by SNSF grant 200021\_228014.}%
\and Maksim Zhukovskii \footremember{alley2}{Department of Computer Science, The University of Sheffield, Sheffield S1 4DP, United Kingdom.}
}

\date{}
\maketitle

\begin{abstract}
Let $K^r_n$ be the complete $r$-uniform hypergraph on $n$ vertices, that is, the hypergraph whose vertex set is $[n] \coloneqq \{1,2,\ldots,n\}$ and whose edge set is $\binom{[n]}{r}$. We form $G^r(n,p)$ by retaining each edge of $K^r_n$ independently with probability $p$. 

An $r$-uniform hypergraph $H\subseteq G$ is \textit{$F$-saturated} if $H$ does not contain any copy of $F$, but any missing edge of $H$ in $G$ creates a copy of $F$. Furthermore, we say that $H$ is \textit{weakly $F$-saturated} in $G$ if $H$ does not contain any copy of $F$, but the missing edges of $H$ in $G$ can be added back one-by-one, in some order, such that every edge creates a new copy of $F$. The smallest number of edges in an $F$-saturated hypergraph in $G$ is denoted by $\sat(G,F)$, and in a weakly $F$-saturated hypergraph in $G$ by $\wsat(G,F)$.

In 2017, Kor\'andi and Sudakov initiated the study of saturation in random graphs, showing that for constant $p$, with high probability $\sat(G(n,p),K_s)=(1+o(1))n\log_{\frac{1}{1-p}}n$, and $\wsat(G(n,p),K_s)=\wsat(K_n,K_s)$. Generalising their results, in this paper, we solve the saturation problem for random hypergraphs $G^r(n,p)$ for cliques $K_s^r$,
for every $2\le r < s$ and constant $p$.
\end{abstract}

\section{Introduction and main results}
Denote by $K^r_n$ the complete $r$-uniform hypergraph on $n$ vertices, that is, the hypergraph whose vertex set is $[n] \coloneqq \{1,2,\ldots,n\}$ and whose edge set is $\binom{[n]}{r}$. For fixed $r$-uniform hypergraphs $F$ and $G$, we say that a hypergraph $H\subseteq G$ is (strongly) $F$-\textit{saturated} in $G$, if $H$ does not contain any copy of $F$, but adding any edge $e\in E(G)\setminus E(H)$ to $H$ creates a copy of $F$. We let $\sat(G, F)$ denote the minimum number of edges in an $F$-saturated hypergraph in $G$. The problem of determining $\sat(K^2_n, K^2_s)$ was raised by Zykov \cite{Z49} in 1949, and independently by Erd\H{o}s, Hajnal, and Moon \cite{EHM64} in 1964. They showed that $\sat(K^2_n, K^2_s)=\binom{n}{2}-\binom{n-s+2}{2}$. Their result was later generalised by Bollob\'as \cite{B65} who showed that $\sat(K^r_n, K^r_s)=\binom{n}{r}-\binom{n-s+r}{r}$. 

We say that an $r$-uniform hypergraph $H\subseteq G$ is \textit{weakly} $F$-saturated in $G$ if $H$ does not contain any copy of $F$, but the edges of $E(G)\setminus E(H)$ admit an ordering $e_1, \ldots, e_k$ such that for each $i\in[k]$, the hypergraph $H_i=H\cup\{e_1,\ldots, e_i\}$ contains a copy of $F$ containing the edge $e_i$. We define $\wsat(G, F)$ to be the minimum number of edges in a weakly $F$-saturated subhypergraph in $G$. Note that the saturation is in fact a restriction of the weak saturation, where one is allowed only to add all the edges at once (simultaneously). Hence, any $H$ which is $F$-saturated in $G$ is also weakly $F$-saturated in $G$, and thus $\wsat(G,F)\le \sat(G,F)$. As such, the weak saturation can be viewed as a natural extension of the saturation. The problem of determining $\wsat(K^r_n, K^r_s)$ was first raised by Bollob\'as \cite{B68}, who conjectured that $\wsat(K^r_n, K^r_s)=sat(K^r_n, K^r_s)$. Using ingenious algebraic methods, this conjecture was verified by Alon \cite{A85}, Frankl \cite{F82}, Kalai \cite{K84, K85}, and Lov\'asz \cite{L77}:
\begin{thm}[\cite{A85,F82,K84,K85,L77}]\label{th: hypergraph}
\begin{align*}
    \wsat(K^r_n, K^r_s)=\sat(K^r_n, K^r_s)=\binom{n}{r}-\binom{n-s+r}{r}.
\end{align*}
\end{thm}

The binomial random graph $G(n,p)$ is obtained by retaining every edge of $K_n$ independently with probability $p$. Similarly, the binomial random $r$-uniform hypergraph $G^r(n,p)$ is obtained by retaining every edge of $K^r_n$ independently with probability $p$. In 2017, Kor\'andi and Sudakov \cite{KS17} initiated the study of the saturation and the weak saturation in random graphs. 
\begin{thm}[\cite{KS17}]
Let $p\in (0,1)$ and let $s$ be a constant. Then, \textbf{whp},
\begin{enumerate}[(a)]
    \item $\sat(G(n,p),K^2_s)=(1+o(1))n\log_{\frac{1}{1-p}}n$.
    \item $\wsat(G(n,p), K^2_s)=\wsat(K^2_n, K^2_s)$.
\end{enumerate}
\end{thm}
Note that there is quite a stark difference between the (strong) $K_s$-saturation number and the weak $K_s$-saturation number in $G(n,p)$. For the weak $K_s$-saturation, \textbf{whp} the answer is the same as in $K_n$, whereas for the $K_s$-saturation \textbf{whp} there is an additional $\log n$ factor compared with $K_s$-saturation in $K_n$. 

Since the work of Kor\'andi and Sudakov, there have been several papers devoted to the study of $\sat(G(n,p),F)$ and $\wsat(G(n,p),F)$ for graphs $F$ which are not cliques, in particular when $p$ is constant. Mohammadian and Tayfeh-Rezaie \cite{MT18} and Demyanov and Zhukovskii \cite{DZ22} proved tight asymptotics for stars, $F=K_{1,s}$. Considering cycles, Demidovich, Skorkin, and Zhukovskii \cite{CycleSat} showed that \textbf{whp} $\sat(G(n,p),C_m)=n+\Theta\left(\frac{n}{\log n}\right)$ for $m\ge 5$, and \textbf{whp} $\sat(G(n,p),C_4)=\Theta(n)$. Considering a more general setting, Diskin, Hoshen, and Zhukovskii \cite{DHZ23} showed that for every graph $F$, \textbf{whp} $\sat(G(n,p),F)=O(n\log n)$, and gave sufficient conditions for graphs $F$ for which \textbf{whp} $\sat(G(n,p),F)=\Theta(n)$. As for the weak saturation, Kalinichenko and Zhukovskii \cite{KZ23} gave sufficient conditions on $F$ for which \textbf{whp} $\wsat(G(n,p),F)=\wsat(K_n,F)$, and Kalinichenko, Miralaei, Mohammadian, and Tayfeh-Rezaie \cite{KMMT23} showed that, for any graph $F$, \textbf{whp} $\wsat(G(n,p),F)=(1+o(1))\wsat(K_n,F)$.

Another natural and challenging direction is to extend the results of Kor\'andi and Sudakov to hypergraphs. Indeed, in their paper from 2017, they asked whether their results could be extended to $r$-uniform hypergraphs. In this paper, we answer that question in the affirmative:
\begin{theorem}\label{th: wsat}
Let $p\in (0,1)$, and let $2\le r<s$ be constants. Then, \textbf{whp},
\begin{enumerate}[(a)]
\item $\wsat(G^r(n,p), K^r_s)=\wsat(K^r_n, K^r_s)$. \label{i: weak}
    \item $\sat(G^r(n,p), K^r_s)=(1+o(1))\binom{n}{r-1}\log_{\frac{1}{1-p^{r-1}}}n$. \label{i: strong}
    \end{enumerate}
\end{theorem}
Similarly to the case of $r=2$, for the weak saturation \textbf{whp} the answer is the same as in the complete $r$-uniform hypergraph, whereas for the saturation there is \textbf{whp} an additional $\log n$ factor compared with the complete $r$-uniform hypergraph. Furthermore, the proof of Theorem~\ref{th: wsat}\ref{i: weak} can be extended to $p\ge n^{-\alpha}$, for some appropriately chosen constant $\alpha>0$. In fact, we believe there exists a threshold for the property of the weak saturation \textit{stability}, that is, for when $\wsat(G^r(n,p),K^r_s)=\wsat(K^r_n, K^r_s)$ --- for graphs, the respective result was obtained by Bidgoli, Mohammadian, Tayfeh-Rezaie, and Zhukovskii \cite{BMTZ20} --- though the problem of finding the threshold itself looks very challenging.

The proof of Theorem \ref{th: wsat}\ref{i: strong} follows the ideas appearing in \cite{KS17}, with some adaptations and more careful treatment. Since the key ideas generalise from the graph setting to the hypergraph setting, we present only an outline of the proof in Section \ref{s: strong}. The full proof is presented in the Appendix.

The weak saturation problem for random hypergraphs turns out to be much harder and requires several new ideas. Indeed, for the upper bound of Theorem \ref{th: wsat}\ref{i: weak}, a novel and delicate construction is needed, see Section \ref{s: weak outline} for an outline of the proof.

The paper is structured as follows. In Section \ref{s: weak outline} we demonstrate the key ideas of the proof of Theorem \ref{th: wsat}\ref{i: weak} by giving a detailed sketch for the case where $r=3$. In Section \ref{s: proof} we prove Theorem \ref{th: wsat}\ref{i: weak}. In Section \ref{s: strong}, we provide a detailed sketch of the proof for Theorem \ref{th: wsat}\ref{i: strong}, with the complete proof appearing in the Appendix. Throughout the paper, we will use the shorthand $G\coloneqq G^r(n,p)$ and abbreviate $r$-uniform hypergraph to $r$-graph.

\section{A detailed sketch of the proof of Theorem \ref{th: wsat}\ref{i: weak} for $r=3$}\label{s: weak outline}
Let us begin with the lower bound. To that end, note that \textbf{whp} all the edges of $\binom{[n]}{r}\setminus E(G)$ can be activated from the edges of $G$. Thus, a weakly $K^r_s$-saturated subgraph of $G$ is \textbf{whp} also a weakly $K^r_s$-saturated of $K^r_n$, and therefore \textbf{whp} $\wsat(G,K^r_s)\ge \wsat(K_n^r,K^r_s)$.

The proof of the upper bound follows from a delicate construction. For the sake of clarity of presentation, and since it already illustrates all the main issues one needs to overcome, we consider the case when $r=3$ in this sketch. 

The natural first step for finding a weakly $K_s^r$-saturated subhypergraph in $G$, similar in spirit to the approach taken in the case when the host graph is a complete hypergraph, is to choose a \textit{core} $C_0\subseteq [n]$ of size $s-3$, such that $G[C_0]\cong K^3_{s-3}$. Observe that, in the case of the complete hypergraph, for every edge $e \subseteq [n] \setminus C_0$ and for every $S\subsetneq e$, we have that $K_n^3[C_0 \cup S]\cong K_{|C_0\cup S|}^3$. In this case, we can then set $H$ to be the graph whose edges are of the form $f\subseteq [n]$ with $f\cap C_0\neq \varnothing$. Then, we may activate all the remaining edges $e\subseteq [n]\setminus C_0$, since in $K_n^3$ we have that $K_n^3[C_0\cup e]\cong K_s^3$, and all the edges induced by $C_0\cup e$, except $e$, are in $H$. This immediately gives the desired upper bound of $\wsat(K_n^3,K_s^3)\le \binom{n}{3}-\binom{n-s+3}{3}$.

However, when the host graph is the random graph $G$, we may have edges $e \subseteq [n] \setminus C_0$ and sets $S \subsetneq e$ such that $G[C_0 \cup S]$ is not a clique. Thus, in what follows, for every such problematic set $S$, we will choose an additional core (that is, a special $(s-3)$-subset of $[n]$).

The proof is divided into four main stages. In the first stage, we define appropriate cores for all relevant subsets of vertices $S\subset[n]\setminus C_0$ of size at most two.
The second stage describes the construction of a weakly $K_s^3$-saturated subhypergraph $H$ of $G$, making use of the properties of the cores established in the first stage. In the third stage, we count the number of edges of $H$. The fourth and final stage proves that $H$ is indeed a weakly $K_s^3$-saturated subhypergraph $G$.

\paragraph{Defining cores.} Our cores will be disjoint subsets of size $s-3$. Let us try to give a rough intuition for the purpose of our cores. For every edge $e$, we aim to find a set $C$ of size $s-3$ such that $G[e\cup C]\cong K_s^3$. Further, we want that either $H[e\cup C]$ will form a clique without $e$, and then $e$ can be immediately activated, or that after several steps of activation of other edges, $H[e\cup C]$ together with the previously activated edges will form a clique without $e$. To allow for this activation process, we will not choose cores for the edges themselves (as this will be inefficient and require too many edges in $H$), but rather assign cores to sets of vertices $S$ of size at most two. Each such core allows us to either activate immediately edges containing $S$, or to do so through some `chain' of activation. As each core will contribute additional edges to $H$, we aim to minimise the number of cores. We now turn to choosing the cores.

We begin by choosing $C_0$, where the only requirement on the set $C_0$ is that $G[C_0]\cong K_{s-3}^3$ (\textbf{whp} such a set exists in $G$). We remark that we will use $C_0$ to activate edges $e\subseteq[n]\setminus C_0$ satisfying that $G[e\cup C_0]=K_s^3$.

We now turn to define a core for every singleton $v\in [n]\setminus C_0$ (noting that cores for different vertices may coincide). For every $v\in [n]\setminus C_0$, if $G[C_0\cup\{v\}]\cong K_{s-2}^3$, we choose the core $C_v\coloneqq C_0$. Otherwise, if $G[C_0\cup \{v\}]\ncong K_{s-2}^3$, we choose a core $C_v$, disjoint from $v$, with the following property.
\begin{align}\label{property:C_v}
    \text{Every edge $f\subseteq C_v\cup \{v\}\cup C_0$ with $f\cap C_v\neq \varnothing$ is in $G$.}
\end{align}
Indeed, \textbf{whp} this is possible for every such $v\in [n]\setminus C_0$.

Let us mention two observations.
\begin{enumerate}[(O\arabic*)]
    \item We have that $G[C_0\cup C_v]\cong K^3_{|C_0\cup C_v|}$, and $|C_0\cup C_v|=s-3$ when $C_v=C_0$, and $|C_0\cup C_v|=2(s-3)$ otherwise. Indeed, every $e\subseteq C_0$ is in $G$ since $G[C_0]\cong K^3_{s-3}$, and every remaining edge $e\subseteq C_0\cup C_v$ with $e\cap C_v\neq \varnothing$ is in $G$ by Property \eqref{property:C_v}. \label{observation:1}
    \item If $u\in C_v$ for some $v$, we have that $G[C_0\cup\{u\}]\cong K_{s-2}^3$, and therefore $C_u=C_0$. \label{observation:2}
\end{enumerate}
We remark that in the final activation step of the argument, for every $v\in [n]$, using the edges activated through $C_0$, we will be able to activate edges of the form $e=\{v, x,y\}\subseteq [n]\setminus C_v$ where $G[e\cup C_v]\cong K_s^3$. Crucially, note that for such edges $e$ (when $C_v\neq C_0$), we have $G[e\cup C_0]\ncong K_s^3$, and thus choosing an additional core was indeed necessary. 

Still, there will be edges $e=\{v_1,v_2,v_3\}$, where for every $i\in [3]$, $G[e\cup C_{v_i}]\ncong K_s^3$. We thus turn to defining cores for pairs of vertices.  We will define cores for all pairs of vertices $S=\{v_1,v_2\}$ which satisfy that $v_1\notin C_{v_2}$ and $v_2\notin C_{v_1}$ (the reason for defining cores for such vertices will become clearer in the last two steps of the argument: when counting the number of edges in $H$ and when considering the activation of edges). We will ensure that $C_S$ satisfies the following property.
\begin{align}\label{property:C_S}
    \text{Every edge $f \subseteq S \cup C_0 \cup C_{v_i} \cup C_S$ with $f \cap C_S \neq \varnothing$ is in $G$, for every $i \in \{1, 2\}$.}    
\end{align}
In all the cases when one of the sets $C_0,C_{v_1},C_{v_2}$ can be chosen for the role of $C_S$ so that Property \eqref{property:C_S} is satisfied, we will indeed set $C_S$ to be equal to such a core. However, in other cases, we will have to assign a new core for $S$. Let us now describe in detail how we define the core $C_S$.

If $G[C_0 \cup \{v_1, v_2\}] \cong K^3_{|C_0\cup\{v_1,v_2\}|}$, then we set $C_{S} = C_0$. Assume now that it is not the case. If for some $i\in\{1,2\}$, every edge $f\subseteq C_{v_i}\cup S\cup C_0$ with $f\cap C_{v_i}\neq \varnothing$ is in $G$, then we set $C_S=C_{v_i}$ (since verifying that \eqref{property:C_S} holds for such a choice of $C_S$ is somewhat technical, we omit the explanation here, and refer to the complete proof in Section \ref{s: proof}). Otherwise, a new choice for the core $C_S$ should be made to later activate edges containing $S$. \textbf{Whp} for any such $S$ we may find a core $C_S$, disjoint from $S$, which satisfies Property \eqref{property:C_S}. 

Let us mention a few observations. Below, we fix $S=\{v_1,v_2\}$.
\begin{enumerate}[(O\arabic*)]\setcounter{enumi}{2}
    \item For every $i\in \{1,2\}$ and $c\in C_S$, we have that $C_{\{v_i,c\}}=C_{v_i}$. Indeed, we need to show that every edge $f\subseteq C_{v_i}\cup \{v_i,c\}\cup C_0$ with $f\cap C_{v_i}\neq \varnothing$ is in $G$. Now, if $c\in f$, then $f\cap C_S\neq \varnothing$ and thus, by Property \eqref{property:C_S}, we have that $f$ is in $G$. Otherwise, $f\subseteq C_{v_i}\cup\{v_i\}\cup C_0$ (and intersects with $C_{v_i}$), and thus, by Property \eqref{property:C_v}, we have that $f$ is in $G$. \label{observation: 3}
    \item For every $i\in \{1,2\}$, we have that $G[C_0\cup C_{v_i}\cup C_{S}]\cong K^3_{|C_0\cup C_{v_i}\cup C_S|}$. Indeed, by Observation \ref{observation:1} we have that $G[C_0\cup C_{v_i}]\cong K^3_{|C_0\cup C_{v_i}|}$, and by Property \eqref{property:C_S}, every edge $f\subseteq C_0\cup C_{v_i}\cup C_S$ with $f\cap C_S\neq \varnothing$ is in $G$. \label{observation: 4}
    \item If $X\subseteq [n]\setminus C_0$ satisfies that $G[C_0\cup X]\cong K^3_{|C_0\cup X|}$, then $C_S=C_0$ for every $S\subseteq X$ of size at most two. In particular, if $Y\subseteq C_{v_i}\cup C_S$ with $|Y|\le 2$, by \ref{observation: 4} we have that $G[C_0\cup Y]\cong K^3_{|C_0\cup Y|}$ and therefore $C_{Y}=C_0$. \label{observation: 5}
\end{enumerate}

\paragraph{Constructing $H$.} Let $H\subseteq G$ be a subhypergraph on $V(G)=[n]$, consisting of the following edges:
\begin{enumerate}[(\arabic*)]
    \item Every edge $e\subseteq C_0$ is in $H$.\label{step c0}
    \item For every $v\in [n]\setminus C_0$, we add to $H$ all edges of the form $\{v\}\cup C'$ for every $C'\subseteq C_v$ of size two. These edges are in $G$ by Property \eqref{property:C_v}. \label{step v and cv}
    \item For every $v\in [n]\setminus C_0$, we add to $H$ all edges of the form $\{v,c_0,c\}$ for every $c_0\in C_0$ and $c\in C_v$. These edges are in $G$ by Property \eqref{property:C_v}. \label{step v, c0 and cv}
    \item For every $S=\{v_1, v_2\}\subseteq [n]$ for which $C_S$ is defined, we add to $H$ all edges of the form $S\cup \{c\}$ for every $c\in C_S$. These edges are in $G$ by Property \eqref{property:C_S}.\label{step v1,v2 and cs}
\end{enumerate}
While Steps \ref{step c0} and \ref{step v and cv} are rather transparent, we refer to Figure \ref{f: sat_core_edges} for an illustration of edges added to $H$ at Steps \ref{step v, c0 and cv} and \ref{step v1,v2 and cs}.
\begin{figure}[H]
\centering
\includegraphics[width=0.4\textwidth]{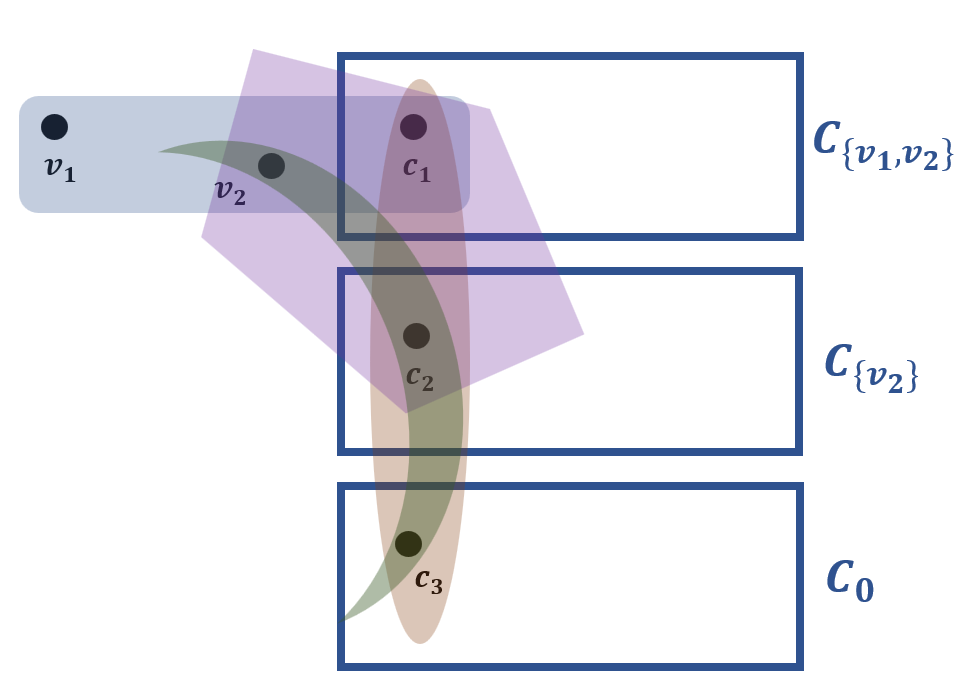}
\caption{Illustration of a \textit{chain} of cores, $C_0$, $C_{v_2}$, and $C_{\{v_1,v_2\}}$, together with edges that were added to $H$. The edge $\{v_1, v_2, c_1\}$ is added to $H$ by \ref{step v1,v2 and cs}. The edge $\{v_2, c_1, c_2\}$ is added to $H$ by \ref{step v1,v2 and cs}, noting that $C_{\{v_2, c_1\}}=C_{v_2}$ by \ref{observation: 3}. The edge $\{c_1, c_2, c_3\}$ is added to $H$ by \ref{step v1,v2 and cs} since $C_{\{c_1, c_2\}}=C_0$ by \ref{observation: 5}. Finally, the edge $\{v_2, c_2, c_3\}$ is added to $H$ by \ref{step v, c0 and cv}.} 
\label{f: sat_core_edges}
\end{figure}

\paragraph{Number of edges in $H$.} The number of edges we have added to $H$ is at most $$\binom{s-3}{3}+\binom{n-s+3}{1}\binom{s-3}{2}+\binom{n-s+3}{2}\binom{s-3}{1}=\binom{n}{3}-\binom{n-s+3}{3}.$$ 
Indeed, the first term comes from the edges induced by $C_0$, which we have added to $H$ at Step~\ref{step c0}.

The second term comes from the edges which were added at Step \ref{step v and cv}. For every $v\in [n]\setminus C_0$, we have chosen $C_v$ (possibly $C_v=C_0$), $|C_v|=s-3$, and added all edges of the form $\{v\}\cup C'$ where $C'\subseteq C_v$ with $|C'|=2$. Thus, we have $\binom{n-s+3}{1}$ choices for $v \in [n] \setminus C_0$ and $\binom{s-3}{2}$ choices for $C' \subseteq C_v$ of size two.

Finally, the third term comes from the edges which were added at Step~\ref{step v, c0 and cv} and at Step~\ref{step v1,v2 and cs}. To see this, consider a pair of vertices $S = \{u, v\} \subseteq [n] \setminus C_0$. Recall that we define $C_S$ only if $v\notin C_{u}$ and $u\notin C_{v}$. Thus, if $C_S$ is not defined, then WLOG $v \in C_{u}$. In this case, since $v \notin C_0$, we must have $C_u \cap C_0 = \varnothing$ (recall that the cores are either identical or vertex-disjoint). In Step~\ref{step v, c0 and cv}, we add all the edges $\{v, u, c_0\}$ for each of the $\binom{s-3}{1}$ possibilities for $c_0 \in C_0$ (note here, that when considering all pairs $\{u,v\}$ in this manner, we in fact cover all edges added at Step \ref{step v, c0 and cv}). Otherwise, $C_S$ is defined, and at Step~\ref{step v1,v2 and cs} we add all edges $\{u, v, c\}$ for each of the $\binom{s-3}{1}$ possibilities for $c \in C_S$.

\paragraph{Activating the edges in $E(G)\setminus E(H)$.} 

First, note that we can activate all edges $e\subseteq [n]\setminus C_0$, such that $G[C_0\cup e]\cong K_s^3$. Indeed, by Observation \ref{observation: 5}, for every non-empty $S \subsetneq e$ we have $C_S = C_0$. Thus, by Steps \ref{step c0}, \ref{step v and cv}, and \ref{step v1,v2 and cs}, all edges induced by $C_0\cup \{e\}$, except $e$, are in $H$. We can thus activate $e$, which closes a copy of $K_s^3$ with $C_0$. See Figure \ref{f: sat_C_0} for illustration.
\begin{figure}[H]
\centering
\includegraphics[width=0.25\textwidth]{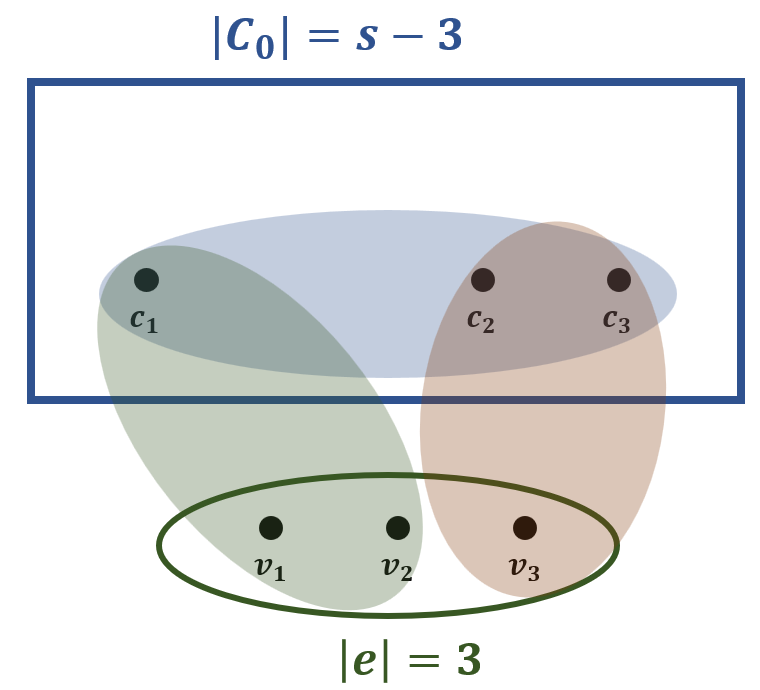}
\caption{The edge $e$ closes a copy with $C_0$ and can thus be activated. The three types of edges that are in $H$ appear in shaded colours. The edges induced by $C_0$ appear in blue, and were added to $H$ at Step \ref{step c0}. The edges that contain one vertex from $e$ and two vertices from $C_0$ appear in red, and were added to $H$ at Step \ref{step v and cv}. The edges that contain two vertices from $e$ and one vertex from $C_0$ appear in green, and were added to $H$ at Step \ref{step v1,v2 and cs}.}
\label{f: sat_C_0}
\end{figure}

We can now activate the remaining edges $e=\{v_1, v_2, v_3\}\subseteq [n]\setminus C_0$. To that end, \textbf{whp} we can choose an \textit{auxiliary clique} $\Tilde{C}\coloneqq \tilde{C}(e)$ disjoint with $e\cup \bigcup_{i,j\in [3]}(C_{v_i}\cup C_{\{v_i,v_j\}})$ (in fact, the union is over $i,j\in [3]$ for which such cores are defined), such that $G[\tilde{C}\cup e]\cong K^3_s$ and for every $i\neq j \in \{1,2,3\}$, every edge $f\subseteq \tilde{C} \cup \{v_i, v_j\}\cup C_{\{v_i, v_j\}}\cup C_{v_i}\cup C_0$ with $f\cap \tilde{C}\neq\varnothing$ is in $G$. We will utilise the following observation. 
\begin{enumerate}[(O\arabic*)]\setcounter{enumi}{5}
    \item For every $c\in \tilde{C}$ and every $i\in [3]$, we have that $c\notin C_{v_i}$ and thus $C_{\{v_i,c\}}$ is defined. Further, $C_{\{v_i,c\}}=C_{v_i}$. To show that, we need to argue that every edge $f\subseteq C_{v_i}\cup \{v_i,c\}\cup C_0$ which intersects with $C_{v_i}$ is in $G$. Indeed, if $c\in f$, then this follows from our choice of $\tilde{C}$, and otherwise, this follows from \eqref{property:C_v}.\label{observation: 6}
\end{enumerate}

Our goal is to show that all edges induced by $e\cup \tilde{C}$, except for $e$, are either in $H$ or can be activated. We will then have that $e$ closes a copy of $K^3_s$ with $\tilde{C}$, and can thus be activated as well. We do so in several activation steps, where some activation steps may depend on previous activation steps (see also Figure \ref{f: sat_edge_chain} for illustration). In what follows, we always assume $i,j\in [3], i\neq j$.
\begin{enumerate}[({A}\arabic*)]
    \item \textbf{Edges $f$ induced by $\tilde{C}\cup C_{\{v_i, v_j\}}\cup C_{v_i}$.} \label{activation 1}

    We claim that such $f$ closes a copy of $K^3_s$ with $C_0$ and can thus be activated. By the choice of $\tilde{C}$ and Observation \ref{observation: 4}, $$G\left[C_0\cup \left(\tilde{C}\cup C_{\{v_i, v_j\}}\cup C_{v_i}\right)\right]\cong K^3_{\left|C_0\cup \left(\tilde{C}\cup C_{\{v_i, v_j\}}\cup C_{v_i}\right)\right|}.$$ Now, consider an edge $f'=\{c_0,x,y\}\subseteq C_0\cup \left(\tilde{C}\cup C_{\{v_i, v_j\}}\cup C_{v_i}\right)$ with $c_0\in C_0$ and $f'\nsubseteq C_0$. We have that, in particular, $G[\{x,y\}\cup C_0]\cong K_{|\{x,y\}\cup C_0|}^3$, and thus $C_{\{x,y\}}=C_0$ by \ref{observation: 5}. Hence, we added $f'$ to $H$ at Step \ref{step v1,v2 and cs} (where $\{x,y\}$ play the role of $S$). Therefore, we can now activate any edge $f\subseteq \tilde{C}\cup C_{\{v_i, v_j\}}\cup C_{v_i}$ since it closes a copy of $K^3_s$ with $C_0$ (together with the edges in $H$).
    \item \textbf{Edges of the form $f=\{v, c_1, c_2\}$ where $v\in \{v_i,v_j\}, c_1,c_2\in C_S\cup \tilde{C}$.}\label{activation 3}

    We claim that $f$ closes a copy of $K^3_s$ with $C_v$ and can thus be activated. Set $S=\{v_i,v_j\}$. We have already activated edges induced by $C_S\cup C_{v}\cup \tilde{C}$ during \ref{activation 1}. In Step \ref{step v and cv}, we added to $H$ all edges of the form $\{v, x_1,x_2\}$ where $x_1,x_2\in C_v$. As for edges $\{v, c, x\}$ where $c\in \{c_1,c_2\}$ and $x\in C_v$, note that by \ref{observation: 3} and \ref{observation: 6}, $C_{\{v,c\}}=C_v$. Thus, $\{v,c,x\}$ is in $H$ by Step \ref{step v1,v2 and cs}. Thus, $f$ closes a copy of $K_s^3$ with $C_v$ (together with edges of $H$ and edges activated during \ref{activation 1}) and may thus be activated. See the left-hand side of Figure~\ref{f: sat_edge_chain} for illustration.
    \item \textbf{Edges of the form $f=\{v_i,v_j,c\}$ where $c\in \tilde{C}$.}\label{activation 4}

    Set $S=\{v_i,v_j\}$. We claim that $f$ closes a copy of $K^3_s$ with $C_S$ and can thus be activated. Edges of the form $\{v_i,v_j,x\}$ where $x\in C_S$ have been added to $H$ at Step \ref{step v1,v2 and cs}. Edges induced by $\{c\}\cup C_S$ were activated in \ref{activation 1}. Edges of the form $\{v, x_1,x_2\}$ where $v\in S$ and $x_1,x_2\in C_S$ were activated in \ref{activation 3}. Furthermore, edges of the form $\{v, c, x\}$ where $v\in S$ and $x\in C_S$ were activated in \ref{activation 3}. Thus, $f$ closes a copy of $K_s^3$ with $C_S$ (together with edges of $H$ and edges activated during \ref{activation 1} and \ref{activation 3}), and may thus be activated. See the right-hand side of Figure~\ref{f: sat_edge_chain} for illustration.
\end{enumerate}
Thus, $e$ closes a copy of $K^3_s$ together with $\tilde{C}$ (with the edges of $H$ and edges activated during \ref{activation 1}-\ref{activation 4}).
\begin{figure}[H]
\centering
\includegraphics[width=0.8\textwidth]{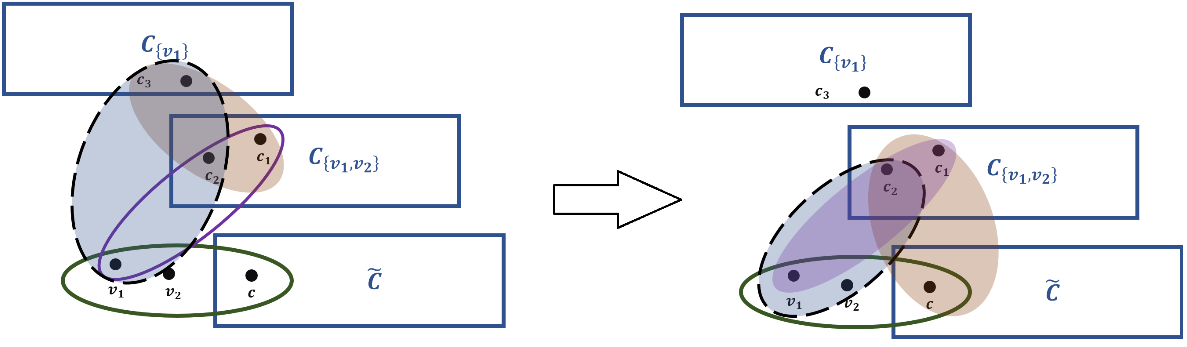}
\caption{Illustration of a \textit{chain} of activation, with the complexity of the construction evident already when $r=3$. Towards activating an edge $\{v_1, v_2, v_3\}$ with $\Tilde{C}$, we need to activate edges of the form $\{v_1, v_2, c\}$ where $c\in \tilde{C}$. To that end, we first activate all edges that form a clique with $C_0$, and in particular, all edges induced by $C_{v_1}\cup C_{\{v_1,v_2\}}\cup \tilde C$. Then, as the left side illustrates, we can activate the edges of the form $\{v_1, c_1, c_2\}$ as they form a clique with $C_{v_1}$. Indeed, the edge $\{c_1, c_2, c_3\}$ forms a clique with $C_0$, and the edge $\{v_1, c_2, c_3\}$ is in $H$ since $C_{\{v_1, c_2\}}=C_{v_1}$ (and was added at Step~\ref{step v1,v2 and cs}). We can then, as the right side illustrates, turn our attention to edges of the form $\{v_1, v_2, c\}$, which will close a clique with $C_{\{v_1, v_2\}}$. Here, the edge $\{c, c_1, c_2\}$ closes a clique with $C_0$ and thus has already been activated, and the edge $\{v_1, v_2, c_2\}$ is in $H$ by Step~\ref{step v1,v2 and cs}.}
\label{f: sat_edge_chain}
\end{figure}

Finally, we are left with edges $e=\{v_1,v_2,v_3\}$ which intersect with $C_0$. The argument for these edges is similar to before, and we omit the details here, referring to the complete proof present in Section \ref{s: proof}.

\section{Weak saturation}\label{s: proof}
Let $r \ge 3$ and $s \ge r + 1$ be integers. Let $p \in (0, 1]$ be a constant. Recall that $G \coloneqq G^r(n,p)$. Set $\ell\coloneqq s-r$.

We start we the lower bound, which is the easiest part of the paper, and follows from the fact that \textbf{whp} $G$ is weakly saturated in $K_n$
\begin{claim}
    \textbf{Whp} $\wsat\left(G, K_s^r\right) \ge \wsat(K_n^r,K_s^r)$.
\end{claim}
\begin{proof}
    Let us first show that \textbf{whp} $G$ is weakly saturated in $K^r_n$. To that end, fix some edge $e\in K^r_n \setminus G$. The probability that for a fixed set $S$ of size $\ell$ we have that $e\cup G[S\cup e]\not\cong K_s^r$ is $1-p^{\binom{s}{r}-1}$. This event is independent for every two disjoint sets of size $\ell$, and thus the probability that $e$ does not form a copy of $K^r_s$ in $G\cup e$ is at most $\left(1-p^{\binom{s}{r}-1}\right)^{\lfloor\frac{n}{\ell}\rfloor}=\exp\left(-\Theta(n)\right)$. The union bound over $\binom{n}{r}\le n^r$ edges $e\in K^r_n$ shows that \textbf{whp} every edge $e\in K^r_n$ belongs to a copy of $K^r_s$ in $G\cup e$, and therefore \textbf{whp} $G$ is weakly saturated in $K^r_n$.

    Now, if $H$ is weakly saturated in $G$, then we can add to $H$ edges one-by-one until we obtain $G$, and then keep adding edges until we reach $K^r_n$. Thus, \textbf{whp} any such $H$ is weakly saturated in $K_n^r$ as well. But then, $\wsat\left(G, K_s^r\right) \ge \wsat(K_n^r,K_s^r)$.
\end{proof}

The rest of this section is devoted to the proof of the upper bound. In Section \ref{s: constructing H}, we lay the groundwork for our proof: define the cores, construct the weakly saturated subhypergraph $H$, and count its edges. In Section \ref{s: edge activation}, we show that there exists an ordering of the edges under which we can activate all edges of $G$ that are not in $H$.

\subsection{Laying the groundwork}\label{s: constructing H}
Before we delve into the fine details, we note that what follows will be a natural extension of the construction in Section \ref{s: weak outline} to general $r$ (together with their formalisation). We will once again define a core $C_0$, and inductively define cores for sets $S\subseteq V(G)\setminus C_0$ with size $1\le |S| \le r-1$. The construction of the cores will naturally have a \textit{chain} property, that is, the core of $S$ may (and will) depend on the cores of $S'\subsetneq S$. Once again, there will be several sets $S$ for which we will not define a core, and instead draw relevant edges with vertices from $C_0$.

We begin by partitioning $V(G)$ into $\lfloor\frac{n}{\ell}\rfloor$ sets of size $\ell$, denoted by $Q_1, \ldots, Q_{\lfloor\frac{n}{\ell}\rfloor}$. Let us further write $\mathcal{Q}\coloneqq\{Q_1, \ldots, Q_{\lfloor\frac{n}{\ell}\rfloor}\}$. Throughout the proof, we will use the following probabilistic lemma.
\begin{lemma}\label{l: prob existence}
Let $k\ge 0$ be a constant. Then, \textbf{whp}, for every $S\subseteq V(G)$ with $|S|=k$, there exists $Q_i$, with $1\le i \le \lfloor\frac{n}{\ell}\rfloor$, such that $S\cap Q_i=\varnothing$ and every edge $e\subseteq Q_i\cup S$ with $e\cap Q_i\neq \varnothing$ is in $G$.
\end{lemma}
\begin{proof}
Fix $S$ and fix $i$ such that $Q_i\cap S=\varnothing$. The probability that every edge $e\subseteq Q_i\cup S$ with $e\cap Q_i\neq \varnothing$ is in $G$ is at least $p^{\binom{\ell+k}{r}}$. There are at least $\lfloor\frac{n}{\ell}\rfloor-k$ different $i$ such that $Q_i\cap S=\varnothing$. Therefore, the probability there doesn't exist such a $Q_i$ is at most 
\begin{align*}
    \left(1-p^{\binom{\ell+k}{r}}\right)^{\lfloor\frac{n}{\ell}\rfloor-k}=\exp\left(-\Theta(n)\right),
\end{align*}
where we used our assumptions that $k, r, \ell$, and $p$ are constants. There are $\binom{n}{k}$ ways to choose $S$, and thus by the union bound, the probability of violating the statement of the lemma is at most
\begin{align*}
    \binom{n}{k}\exp\left(-\Theta(n)\right)=o(1),
\end{align*}
as required.
\end{proof}

\paragraph{Defining the cores.} Let $i_0$ be the first index in $\left[\lfloor\frac{n}{\ell}\rfloor\right]$ such that $G[Q_{i_0}]\cong K^r_{\ell}$, and let us set $C_0\coloneqq Q_{i_0}$ (note that by Lemma \ref{l: prob existence} \textbf{whp} such an $i_0$ exists). For every $j>0$, let $i_j$ be the $j$-th index in $\left[\lfloor\frac{n}{\ell}\rfloor\right]$ for which we have that $G[C_0\cup Q_{i_j}]\cong K^r_{2\ell}$ (that is, the $j$-th occurrence of some $Q\in \mathcal{Q}$ for which it holds). We then set $C_j\coloneqq Q_{i_j}$ (once again, note that since $G[C_0]\cong K^r_{\ell}$, by Lemma \ref{l: prob existence} \textbf{whp} such a $C_j$ exists). We call these sets \textit{cores}, and we enumerate them $C_0, C_1,\ldots, C_m$. Note that for every $i\neq j\in [m]$, $C_i\neq C_j$, as they are two different $Q\neq Q'\in \mathcal{Q}$. We continue assuming these cores have been defined deterministically.

\paragraph{Assigning cores to sets.} Set $C_{\varnothing}=C_0$. For every vertex $v \in V \setminus C_0$, let $i(v)$ be the first index such that the following holds.
\begin{enumerate}
    \item Every edge $e \subseteq \{v\} \cup C_0 \cup C_{i(v)}$ with $e \cap C_{i(v)} \neq \varnothing$ is in $G$.
    \item $\{v\} \cap C_{i(v)} = \varnothing$.
\end{enumerate}
We then set $C_{\{v\}} = C_{i(v)}$. Note that by the properties of the cores and by Lemma \ref{l: prob existence}, \textbf{whp} such a $C_{\{v\}}$ exists for every $v\in V(G)\setminus C_0$. Furthermore, observe that the properties of $C_{\{v\}}$ imply that $G[\{v\}\cup C_{\{v\}}]\cong K^r_{\ell+1}$.

Now, we define cores for suitable subsets outside of $C_0$ by induction on their size. For every $j\in [2, r-1]$ and for every $S \subseteq V(G) \setminus C_0$ of size $j$, if
\begin{align}\label{core-definable}
    \text{there is no } S'\subsetneq S, \text{ such that } C_{S'} \text{ is defined and } S\setminus S'\subseteq C_{S'}, \tag{$S$ is core-definable}
\end{align}
we define $C_S$ in the following way. Let $i(S)$ be the first index in $[m]$ such that the following holds for every integer $t\in [j-1]$ and a sequence $\varnothing=S_0 \subsetneq \dots \subsetneq S_t \subsetneq S$ for which $C_{S_0}, \ldots, C_{S_t}$ are defined.
\begin{enumerate}[(P\arabic*)]
\item Every edge $e \subseteq S \cup C_{i(S)} \cup C_{S_0} \cup \dots \cup C_{S_t}$ with $e \cap C_{i(S)} \neq \varnothing$ is in $G$. \label{p: edges intersecting C_S}
\item $G\left[C_{i(S)} \cup C_{S_0} \cup \dots \cup C_{S_t} \right]\cong K^r_{|C_{i(S)} \cup C_{S_0} \cup \dots \cup C_{S_t}|}$.
\label{p: small S_0 and cores are clique}
    \item $S \cap C_{i(S)} = \varnothing$. \label{p: no intersection}
\end{enumerate}
We then set $C_S = C_{i(S)}$. Note that we allow $S_0=\varnothing$ in order to include edges intersecting with $C_{\varnothing}=C_0$.

Note that the condition in Property \ref{p: edges intersecting C_S} applies only to edges intersecting with $C_{i(S)}$. Furthermore, by induction, we have that $G\left[S_0 \cup C_{S_0} \cup \dots \cup C_{S_t} \right]\cong K^r_{|S_0 \cup C_{S_0} \cup \dots \cup C_{S_t}|}$, and thus the condition in Property \ref{p: small S_0 and cores are clique} also concerns only edges intersecting with $C_{i(S)}$. Thus, by these properties and by Lemma \ref{l: prob existence}, \textbf{whp} for every $S\subseteq V(G)\setminus C_0$ of size $j$ such that there is no $S'\subsetneq S$ with $C_{S'}$ defined and $S\setminus S'\subseteq C_{S'}$, we can find such a $C_S$. In what follows, we assume this holds deterministically.

Note that we have defined a core $C_{S}$ for every set $S\subseteq V(G)\setminus C_0$ of size at most $r-1$ which is \hyperref[core-definable]{core-definable}. Furthermore, if $C_{S}$ is not defined, then there exists $S'\subseteq S$ for which $C_{S'}$ is defined and $S\setminus S'\subseteq C_{S'}$. Finally, note that
\begin{align}\label{align:i(S') le i(S)}
    i(S') \le i(S) \quad \text{for every $S' \subseteq S$ such that $C_{S'}$ and $C_S$ are defined.}
\end{align}
Indeed, the Properties \ref{p: edges intersecting C_S} through \ref{p: no intersection} are closed under inclusion, and thus any index satisfying the properties for $S$ must already satisfy the properties for $S'$.

\paragraph{Constructing $H$.} Let $H \subseteq G$ be a subhypergraph consisting of the following edges:
\begin{enumerate}[label=(E\arabic*)]
    \item\label{edges-from-C_0} Every edge $e \subseteq C_0$ is in $H$.
    \item\label{edges-from-C_S} For every $\varnothing\neq S \subseteq V(G) \setminus C_0$ for which $C_S$ is defined, we add to $H$ all edges of the form $S \cup C'$ for every $C' \subseteq C_S$ of size $r - |S|$.
    \item\label{edges-from-C_0-and-C_S} For every $\varnothing\neq S \subseteq V(G) \setminus C_0$ for which $C_S$ is defined, we add to $H$ all edges of the form $S \cup C' \cup C'_0$ for every $\varnothing\neq C' \subseteq C_S$ and $\varnothing\neq C'_0 \subseteq C_0$ satisfying $|C' \cup C'_0| = r - |S|$.
\end{enumerate}
Note that edges of type \ref{edges-from-C_0} are in $G$ by the definition of the cores, and edges of type \ref{edges-from-C_S} and type \ref{edges-from-C_0-and-C_S} are in $G$ by Property \ref{p: edges intersecting C_S} (and the first property when defining cores for singletons).

\paragraph{Number of edges in $H$.} Let us bound from above the number of edges added in each step of the construction. 

First, at Step \ref{edges-from-C_0}, we add $\binom{\ell}{r}$ edges to $H$ induced by $C_0$. We now turn to Steps \ref{edges-from-C_S} and \ref{edges-from-C_0-and-C_S}. Let us set
\begin{align*}
    A &= \left\{S \colon \varnothing\neq S \subseteq V(G) \setminus C_0 \wedge C_S \textit{ is defined}\right\} \quad \text{and} \\
    B &= \left\{S \cup C' \colon \varnothing \neq S \subseteq V(G) \setminus C_0 \wedge C_S \textit{ is defined} \wedge C' \subseteq C_S\right\}.
\end{align*}
We have the following.
\begin{itemize}
    \item In Step \ref{edges-from-C_S}, we add the edges $S \cup C'$ for every $S \in A$ and $C' \subseteq C_S$ of size $r - |S|$.
    \item In Step \ref{edges-from-C_0-and-C_S}, we add the edges $X \cup C'_0$ for every $X \in B$ and $C'_0 \subseteq C_0$ of size $r - |X|$.
\end{itemize}

Thus, the number of edges considered at Step \ref{edges-from-C_S} is at most
\[
    \sum_{S \in A} \binom{\ell}{r-|S|} = \sum_{\substack{S\subseteq V(G)\setminus C_0\\ C_S\text{ is defined}}}\binom{\ell}{r-|S|}.
\]
Note that, for every $S \subseteq V(G) \setminus C_0$ such that $C_S$ is defined and for every $\varnothing \neq C' \subseteq C_S$, we have that $C_{S \cup C'}$ is not defined. Indeed, in that case, we have that $S \subseteq S \cup C'$, $C_{S}$ is defined, and $(S \cup C') \setminus S \subseteq C_S$. Thus, by definition, $S \cup C'$ is not \hyperref[core-definable]{core-definable}. Hence, we have that
\[
    B \subseteq \{X\subseteq V(G)\setminus C_0 \colon C_X\text{ is not defined}\}.
\]
Moreover, the number of added edges at Step \ref{edges-from-C_0-and-C_S} is at most
\[
    \sum_{\substack{X\subseteq V(G)\setminus C_0\\ C_X\text{ is not defined}}}\binom{\ell}{r-|X|}.
\]

Therefore, the number of edges in $H$ is at most
\begin{align*}
     \binom{\ell}{r}+\sum_{i=1}^{r-1}\left(\sum_{\substack{S\subseteq V(G)\setminus C_0, \quad |S|=i\\ C_S\text{ is defined}}}\binom{\ell}{r-i}+\sum_{\substack{X\subseteq V(G)\setminus C_0, \quad|X|=i\\ C_X\text{ is not defined}}}\binom{\ell}{r-i}\right)\\
     =\binom{\ell}{r}+\sum_{i=1}^{r-1}\binom{n-\ell}{i}\binom{\ell}{r-i}=\binom{n}{r}-\binom{n-\ell}{r}.
\end{align*}

\subsection{Activating the remaining edges}\label{s: edge activation}
The argument for activating the remaining edges will be a natural extension of the argument given for the case $r=3$ in Section \ref{s: weak outline}, together with its formalisation.

\paragraph{Assigning an auxiliary core to every edge.} Let us fix an edge $e\in E(G)$. Note that we only assign cores to sets of size at most $r-1$, that is, there is no $C_e$ defined. By Lemma \ref{l: prob existence} applied with $$S=e\cup \bigcup_{\substack{A\subsetneq e\\ A\text{ is \hyperref[core-definable]{core-definable}}}}C_A,$$ we can choose an auxiliary core $\tilde{C}\coloneqq \tilde{C}(e)$ from $Q_1,\ldots, Q_{\lfloor\frac{n}{\ell}\rfloor}$, such that the following holds. For every $S\subsetneq e$, every $t\in [0,|S|-1]$, and every $S_0\subsetneq \cdots\subsetneq S_t\subsetneq S$ such that $C_S$ is defined and $C_{S_i}$ is defined for every $i\in[0,t]$:
\begin{enumerate}[($\overline{P}$\arabic*)]
    \item Every edge $f \subseteq S \cup C_S \cup C_{S_0} \cup \cdots \cup C_{S_t}\cup \tilde{C}$ with $f \cap \tilde{C} \neq \varnothing$ is in $G$. \label{tilde p: edges intersecting C_S}
    \item $G\left[S_0\cup C_{S} \cup \tilde{C}\cup C_{S_0} \cup \cdots \cup C_{S_t} \right]\cong K^r_{|S_0\cup C_{S} \cup \tilde{C}\cup C_{S_0} \cup \cdots \cup C_{S_t}|}$.
    \label{tilde p: small S_0 and cores are clique}
    \item $(S \cup C_{S}\cup C_{S_0}\cup\cdots\cup C_{S_t})\cap \tilde{C} = \varnothing$. \label{tilde p: no intersection}
\end{enumerate}
Lemma \ref{l: prob existence} guarantees the (typical) existence of edges intersecting with $\tilde{C}$. Indeed, the above properties concern only edges intersecting with $\tilde{C}$. To see that in Property \ref{tilde p: small S_0 and cores are clique}, note that $G\left[S_0\cup C_{S} \cup C_{S_0} \cup \cdots \cup C_{S_t} \right]\cong K^r_{|S_0\cup C_{S} \cup \tilde{C}\cup C_{S_0} \cup \cdots \cup C_{S_t}|}$ by Property \ref{p: small S_0 and cores are clique} and thus Property \ref{tilde p: small S_0 and cores are clique} indeed concern only edges intersecting with $\tilde{C}$. Lemma \ref{l: prob existence} further guarantees the non-intersection of $\tilde{C}$ with $S$, which yields Property \ref{tilde p: no intersection}.

\paragraph{Core extension property.} Recall that, by \eqref{align:i(S') le i(S)}, given $S\subseteq S\cup U\subseteq V(G)$ such that $C_S$ and $C_{S\cup U}$ are defined, then $i(S)\le i(S\cup U)$. If $U$ satisfies $i(S\cup U)=i(S)$, then $C_{S\cup U}=C_S$. In the following claim, we show that for a given $S$ for which $C_S$ is defined, we have that $U$ is satisfies $i(S\cup U)=i(S)$ whenever $U\subseteq C_{S_0}\cup \cdots \cup C_{S_t}$, for some $S\subsetneq S_0\subsetneq\dots\subsetneq S_t$. This `extension' property of the cores will be important for us throughout the activation process, and in fact, we required Properties \ref{p: edges intersecting C_S} and \ref{p: small S_0 and cores are clique} when choosing our cores so that this extension property will hold.
\begin{claim} \label{c: extending our cores}
    Let $e_0$ be an edge in $G$. Let $S \subseteq V(G) \setminus C_0$ be such that $C_S$ is defined. For every set $U$ and $t\in [r-1]$ such that $|S\cup U|\le r-1$ and $U \subseteq C_{S_0} \cup \cdots \cup C_{S_t}\cup \tilde{C}(e_0)$ for some $S \subsetneq S_0 \subsetneq \dots \subsetneq S_t\subsetneq e$ for which $C_{S_0}, \dots, C_{S_t}$ are defined and different than $C_{S}$, we have that
    \[
        C_{S \cup U} = C_{S}.
    \]
\end{claim}

\begin{proof}
We prove by induction on $|S\cup U|$, where the case when $|S\cup U|=0$ follows trivially.

Let $S'\cup U'\subsetneq S\cup U$ with $S'\subseteq S$ and $U'\subseteq U$ such that $C_{S'\cup U'}$ is defined. Suppose towards contradiction that $S'$ is not \hyperref[core-definable]{core-definable}. Then, there exists $S''\subsetneq S'$ such that $C_{S''}$ is defined and $S'\setminus S''\subseteq C_{S''}$. But then, by the induction hypothesis, $C_{S''\cup U'}=C_{S''}$, and by our assumption $(S'\cup U')\setminus (S''\cup U')\subseteq C_{S''}=C_{S''\cup U'}$. Therefore, $C_{S'\cup U'}$ is not defined --- contradiction. Therefore, $C_{S'}$ is defined, and by the induction hypothesis, we have that $C_{S'\cup U'}=C_{S'}$.

Let us verify that $C_{S\cup U}$ is defined. Suppose towards contradiction that $S\cup U$ is not \hyperref[core-definable]{core-definable}. Then, there exists $S'\cup U'\subsetneq S\cup U$ such that $S'\subseteq S$, $U'\subseteq U$, $C_{S'\cup U'}$ is defined, and $(S\cup U)\setminus (S'\cup U')\subseteq C_{S'\cup U'}$. We then have by the above that $C_{S'\cup U'}=C_{S'}$. Suppose first that $U'=U$. Then $(S\cup U)\setminus (S'\cup U')=S\setminus S'\subseteq C_{S'}$, contradicting the fact that $C_{S}$ is defined. Otherwise, $U'\subsetneq U$. Then, by the induction hypothesis, $C_{S\cup U'}=C_S$. By our assumption, $(S\cup U')\setminus (S'\cup U')\subseteq C_{S'\cup U'}$, contradicting the fact that $C_{S\cup U'}$ is defined. Therefore, we conclude that $C_{S\cup U}$ is defined.

Therefore, it suffices to verify that $S\cup U$ satisfies Properties \ref{p: edges intersecting C_S} through \ref{p: no intersection} with respect to $C_S$. Let $k\ge1$ be an integer and let $A_1\subsetneq A_2\subsetneq\cdots\subsetneq A_k\subsetneq S\cup U$ be such that $C_{A_i}$ is defined for every $i\in [k]$. We can then write $A_i=(A_i\cap S)\cup (A_i\cap U)$. Noting that $|A_i|<|S\cup U|$, by the above and by the hypothesis, we then have that $C_{A_i}=C_{S\cap A_i}$, and therefore when we consider $C_{A_i}$, we may assume that $A_i\subseteq S$. 

\begin{enumerate}
    \item First, let us show that every edge $e\subseteq (S\cup U)\cup C_S\cup C_{A_1}\cup \cdots \cup C_{A_k}$ with $e\cap C_S\neq \varnothing$ is in $G$, that is, let us verify Property \ref{p: edges intersecting C_S}. Note that, by the above, there is some sequence of $S_1'\subsetneq \cdots \subsetneq S_k'\subsetneq S$ such that $C_{A_j}=C_{S_j'}$ for every $j\in [k]$. Since $S\subsetneq S_0$, we have that $B_1\subsetneq \cdots \subsetneq B_m$ where $m=k+t+2$ and $B_1=S_1', \ldots, B_k=S_k', B_{k+1}=S, B_{k+1}=S_0, \ldots, B_m=S_t$. Thus, if $e\cap \tilde{C}\neq \varnothing$, then by Property \ref{tilde p: edges intersecting C_S} with $B_m$ we have that $e$ is in $G$. Otherwise, let $\tau\in [t]$ be the maximal index such that $e\cap C_{S_{\tau}}\neq \varnothing$. We may assume that $\tau\ge 1$, otherwise $e\subseteq S\cup C_S\cup C_{A_1}\cup \ldots \cup C_{A_k}$ and then the above follows from Property \ref{p: edges intersecting C_S} with respect to $S$ and $C_S$. Then, choosing $B_1=A_1, \ldots, B_k=A_k, B_{k+2}=S, B_{k+3}=S_{1},\ldots, B_m=S_{\tau}$, since $e\cap C_{S_{\tau}}\neq\varnothing$, we have that $e$ is in $G$ by Property \ref{p: edges intersecting C_S}. Therefore, $S\cup U$ satisfies property \ref{p: edges intersecting C_S} with respect to $C_S$.

    \item We now turn to show that $G[A_1\cup C_{S}\cup C_{A_1}\cup \cdots \cup C_{A_k}]\cong K^r_{|A_1\cup C_{S}\cup C_{A_1}\cup \cdots \cup C_{A_k}|}$. It suffices to show that $G[(A_1\cup U)\cup C_{S}\cup C_{A_1}\cup \cdots \cup C_{A_k}]\cong K^r_{|(A_1\cup U)\cup C_{S}\cup C_{A_1}\cup \cdots \cup C_{A_k}|}$, where we may assume, as discussed above, that $A_i\subseteq S$ for all $i\in [k]$.

    To that end, note that by Property \ref{p: small S_0 and cores are clique} (or by Property \ref{tilde p: small S_0 and cores are clique} if $U$ intersects with $\tilde{C}$) with respect to $S_t$ and $C_{S_t}$, we have that for every $B_1\subsetneq B_2\subsetneq \cdots \subsetneq B_m\subseteq S_t$ such that $C_{B_i}$ is defined for every $i\in [m]$, $G[B_1\cup C_{S_t}\cup C_{B_1}\cup\cdots\cup C_{B_m}\cup\tilde{C}]\cong K^r_{|B_1\cup C_{S_t}\cup C_{B_1}\cup\cdots\cup C_{B_m}\cup\tilde{C}|}$. We may thus choose $B_1=A_1, \ldots, B_k=A_k, B_{k+1}=S, B_{k+2}=S_0,\ldots, B_m=S_t$, and since $U\subseteq C_{S_0}\cup\cdots\cup C_{S_t}\cup\tilde{C}$, we conclude that $S\cup U$ satisfies Property \ref{p: small S_0 and cores are clique}.

    \item Finally, let us show that $(S\cup U)\cap C_S=\varnothing$. Indeed, since $C_S$ is defined, we have that $S\cap C_S=\varnothing$, and by our assumption, $C_{S_0},\ldots, C_{S_t}$ are different from $C_{S}$ (and in particular disjoint, due to the definition of the sequence $Q_i$, $i\in\left[\lfloor\frac{n}{\ell}\rfloor\right]$), and by property \ref{tilde p: no intersection}, $\tilde{C}$ is disjoint from $C_S$. Therefore, $U\cap C_S=\varnothing$ and $(S\cup U)\cap C_S=\varnothing$.
\end{enumerate}
\end{proof}

\paragraph{Activating the edges.} With this at hand, we are ready to prove that we can activate all the edges of $G$ outside of $H$. The argument here will be a natural extension, and formalisation, of the argument for $r=3$ detailed in Section \ref{s: weak outline}. We consider separately edges that intersect with $C_0$ and edges that do not. 

\begin{claim}\label{claim:activate-edges-outside-C0}
    One can activate all edges $e \in V(G) \setminus C_0$. 
\end{claim}
\begin{proof}
    Let $e \subseteq V(G) \setminus C_0$ be of size $r$. We will show that $e$ closes a copy of $K^r_s$ with $\Tilde{C}=\tilde{C}(e)$.

    Let us prove by induction on $|S|$ that for every set $S \subsetneq e$, one can activate all edges (that are not in $H$) of the form $S \cup U$, where $U \subseteq \Tilde{C} \cup C_{S_0} \cup \dots \cup C_{S_t}$ for $S \subsetneq S_0 \subsetneq S_1 \subsetneq \dots \subsetneq S_t \subsetneq e$ for which $C_{S_i}$ is defined for all $i\in[t]$. Note that we need the above only for $U\subseteq \tilde{C}$, however, for the induction argument it will be easier to prove the aforementioned stronger statement.

    For the base case of the induction, we consider $S = \varnothing$. Let $f \subseteq \Tilde{C} \cup C_{S_0} \cup \dots \cup C_{S_t}$ be an edge for some $\varnothing \subsetneq S_0 \subsetneq S_1 \subsetneq \dots \subsetneq S_t \subsetneq e$. Note that by Property \ref{p: small S_0 and cores are clique}, $G[f\cup C_0]\cong K^r_{|f\cup C_0|}$, and thus if $f\cap C_0\neq \varnothing$, we have that $f$ is in $H$. Let us thus suppose that $f\cap C_0=\varnothing$. By \ref{edges-from-C_0}, all the edges induced by $C_0$ are in $H$. Moreover, for every $S' \subsetneq f$, since $G[f\cup C_0]\cong K^r_{|f\cup C_0|}$, we have that $C_{S'} = C_0$. Thus, by \ref{edges-from-C_S}, every edge of the form $S' \cup C'_0$, for every $C'_0 \subseteq C_0$ of size $r - |S'|$, is in $H$. Hence, $f\cup C_0$ is a clique in $G$, and all its edges but $f$ are in $H$. Thus, we can activate $f$.

    For the induction step, let $i \in [r-1]$. Suppose that for every $S'\subsetneq e$, $|S'|<i$, and $S'\subsetneq A_1\subsetneq A_2\subsetneq\dots\subsetneq A_k\subsetneq e$ such that $C_{A_i}$ exists for all $i\in[k]$, we have activated (or added to $H$) all edges of the form $S'\cup U'$ where $U'\subseteq \tilde{C}\cup C_{A_1}\cup\ldots\cup C_{A_k}$. Let $S \subsetneq e$ of size $i$. Let $f$ be an edge of the form $S \cup U$ where $U \subseteq \Tilde{C} \cup C_{S_0} \cup \dots \cup C_{S_t}$ for some $S \subsetneq S_0 \subsetneq S_1 \subsetneq \dots \subsetneq S_t \subsetneq e$ for which $C_{S_0}, \dots, C_{S_t}$ are defined. We will consider two separate cases, determined by whether $C_{S}$ is defined or not.

    Assume first that $C_{S}$ is defined. We will show that we can activate $f$ by closing a copy of $K^r_s$ induced by $f \cup C_{S}$. Indeed, note if $f\cap C_S\neq\varnothing$ then $f$ is in $H$ (indeed, by Claim \ref{c: extending our cores}, $C_{f\setminus C_{S}}=C_S$). We may thus suppose $f\cap C_S=\varnothing$. By Claim \ref{c: extending our cores}, we have that $C_{S \cup U'} = C_{S}$ for every $U' \subseteq U$. Thus, we added to $H$ all edges of the form $S \cup U' \cup C'$ for every $U' \subsetneq U$ and $C' \subseteq C_S$. By the induction hypothesis, we have already activated (or have in $H$) all edges of the form $S' \cup U'$ for every $S' \subsetneq S$, $U' \subseteq U \cup C_S$. Hence, we can activate $f$ by closing a copy of $K^r_s$ induced by $f\cup C_S$.

    Assume now that $C_{S}$ is not defined. Then, there exists $S' \subsetneq S$ such that $C_{S'}$ is defined and $S \setminus S' \subseteq C_{S'}$. By Claim \ref{c: extending our cores}, we have that $C_{S' \cup U'} = C_{S'}$ for every $U' \subseteq U$. Hence, we have already added to $H$ all edges of the form $S' \cup U' \cup C'$ for every $U' \subseteq U$ and $\varnothing\neq C' \subseteq C_{S'}$. In particular, letting $C' = S \setminus S'$, we get that the edge $f = S \cup U$ is already in $H$.
    
    Altogether, all edges induced by $e\cup \tilde{C}$, other than $e$, are either in $H$ or have been activated, and thus $e$ can be activated by closing a copy of $K^r_s$ with $\tilde{C}$. 
\end{proof}

\begin{claim}
    One can activate all the edges intersecting with $C_0$.
\end{claim}
\begin{proof}
    Let $e$ be an edge such that $e \cap C_0 \neq \varnothing$. Once again, we will show that $e$ closes a copy of $K^r_s$ with $\Tilde{C}=\tilde{C}(e)$ in $H$.

    Let us prove by induction on $|S|$ that for every set $S \subsetneq e \setminus C_0$, one can activate (if not already in $H$) all edges of the form $S \cup U$ where $U \subseteq C_0 \cup \Tilde{C} \cup C_{S_0} \cup \dots \cup C_{S_t}$ for some $S \subsetneq S_0 \subsetneq S_1 \subsetneq \cdots \subsetneq S_t \subsetneq e \setminus C_0$ for which $C_{S_0},\ldots, C_{S_t}$ were defined. 
    We can then activate $e$ as it closes a copy of $K^r_S$ with $\tilde{C}$.

    For the base case of the induction, we consider $S = \varnothing$. Let $f \subseteq C_0 \cup \Tilde{C} \cup C_{S_0} \cup \dots \cup C_{S_t}$ be an edge for some $S \subsetneq S_0 \subsetneq S_1 \subsetneq \cdots \subsetneq S_t \subsetneq e \setminus C_0$ for which $C_{S_0},\ldots, C_{S_t}$ were defined. By Claim \ref{claim:activate-edges-outside-C0}, we may assume that $f \cap C_0 \neq \varnothing$. By Claim \ref{c: extending our cores}, we have that $C_{f\setminus C_0}=C_{S}=C_0$ since $S=\varnothing$. Therefore, by \ref{edges-from-C_S}, $f$ is in $H$. 
    
    For the induction step, let $i \in [r-1]$ and take $S \subsetneq e \setminus C_0$ of size $i$. Let $f$ be an edge of the form $S \cup U$, where $U \subseteq C_0 \cup \Tilde{C} \cup C_{S_0} \cup \dots \cup C_{S_t}$ for some $S \subsetneq S_0 \subsetneq S_1 \subsetneq \cdots \subsetneq S_t \subsetneq e \setminus C_0$ for which $C_{S_0},\ldots, C_{S_t}$ were defined. We will consider two separate cases, determined by whether $C_{S}$ is defined or not.

    Assume first that $C_{S}$ is defined. If $U \cap C_0 = \varnothing$, then we can activate $S \cup U$ by Claim \ref{claim:activate-edges-outside-C0}. If $U \cap C_0 \neq \varnothing$, note that by Claim \ref{c: extending our cores}, $C_{S \cup (U \setminus C_0)} = C_S$. We thus have that $f$ closes a copy of $K^r_s$ with $C_S$. Indeed, let $h \subseteq f \cup C_S$ be an edge with $h\neq f$. If $S \cap h \neq S$, then $|S\cap h|\le i-1$. Writing $S'=S\cap h$ and $U'=h\setminus S$, we want to apply the induction hypothesis with $S'\cup U'$. Indeed, note that $U'\subseteq C_0 \cup \Tilde{C} \cup C_S \cup C_{S_0} \cup \dots \cup C_{S_t}$, and in particular where $S'\subsetneq S \subsetneq S_0\subsetneq \cdots \subsetneq S_t\subsetneq e$. Thus, by the induction hypothesis, we have already activated (or added to $H$) the edge $h$. Otherwise, we have $S \subseteq h$. Noting that by Claim \ref{c: extending our cores}, $C_{h \setminus C_0} = C_S$, by \ref{edges-from-C_0-and-C_S} we have already added the edge $h$ to $H$.

    Assume now that $C_{S}$ is not defined. Then, there exists $S' \subseteq S$ such that $C_{S'}$ is defined and $S \setminus S' \subseteq C_{S'}$. By Claim \ref{c: extending our cores}, $C_{S' \cup (U \setminus C_0)} = C_{S'}$. Thus, by \ref{edges-from-C_0-and-C_S}, we added to $H$ all edges of the form $S' \cup (U \setminus C_0) \cup C' \cup C'_0$ for every $\varnothing\neq C' \subseteq C_{S'}$ and $C'_0 \subseteq C_0$. In particular, letting $C' = S \setminus S' \subseteq C_{S'}$ and $C'_0 = U \cap C_0$, we get that the edge $f = S' \cup (U \setminus C_0) \cup (S \setminus S') \cup (U \cap C_0)=S\cup U$ was added to $H$.
\end{proof}

\section{Strong Saturation}\label{s: strong}
We begin with the proof of the lower bound of Theorem \ref{th: wsat}\ref{i: strong}, which sheds some light as to why typically $\sat(G^r(n,p),K^r_s)=(1+o(1))\binom{n}{r-1}\log_{\frac{1}{1-p^{r-1}}}(n)$.

\begin{proof}[Proof of the lower bound of Theorem \ref{th: wsat}\ref{i: strong}]
Let us show that if $H$ is $K_s^r$-saturated in $G$, then $e(H)\ge (1+o(1))\binom{n}{r-1}\log_{\frac{1}{1-p^{r-1}}}n$. Note that if $H$ is $K_s^r$-saturated in $G$, then adding any $e\in E(G)\setminus E(H)$ creates a new copy of $K_s^r$, and in particular, a new copy of $K_{r+1}^r$. Let $\alpha=\frac{1}{1-p^{r-1}}$.  

Given an $(r-1)$-subset $S\subseteq V(G)$, let $N_H(S)=\{v\in V(G)\colon \{v\}\cup S\in E(H)\}$. Let
\begin{align*}
    A\coloneqq \{S\subseteq V(G)\colon |S|=r-1 \text{ and } |N_H(S)|\ge \log_{\alpha}^2n\},
\end{align*}
and set $B=\binom{V(G)}{r-1}\setminus A$. Note that if $|A|=\Omega\left(n^{r-1}\right)$, then $|E(H)|=\Omega\left(n^{r-1}\log_{\alpha}^2n\right)$, and we are done. We may thus assume that $|A|=o\left(n^{r-1}\right)$, and thus $|B|=(1+o(1))\binom{n}{r-1}$.

Let $S\in B$. We claim that there are at least $(1+o(1))\log_{\alpha}n$ edges $e\in E(H)$, such that $S$ is the only $(r-1)$-subset of $e$ which is in $B$. This will imply that the number of edges in $H$ is at least $|B|\log_{\alpha}n=(1+o(1))\binom{n}{r-1}\log_{\alpha}n$, as required. Let $u\in V(G)$ be such that $e=\{u\}\cup S\in E(G)\setminus E(H)$. Then $\{u\}\cup S$ closes a copy of $K_{r+1}^r$ together with $H$, in particular, there is some $w\in V(G)$ such that $e\cup H[S\cup \{u\}\cup \{w\}]\cong K_{r+1}^r$. Note that for all but at most $\log_{\alpha}^4n$ choices of $u$, we have that the only $(r-1)$-subset of $S\cup \{w\}$ that is in $B$ is $S$. Indeed, $S\in B$, so there are at most $\log_{\alpha}^2n$ choices of $w$ such that $\{w\}\cup S\in E(H)$. Then, if at least one other $(r-1)$-subset $S\neq S'\subseteq S\cup\{w\}$ is such that $S'\in B$, then we have at most $\log_{\alpha}^2n$ choices for this $u$ as well, since $S'\cup \{u\}\in E(H)$. 

Thus, we have a set $U$ of at least $|N_G(S)|-\log_{\alpha}^4-O(1)$ vertices $u$ such that there exists $w=w(u)$ forming an edge with $S$ and with every $\{u\}\cup S'$, where $S'\subset S$, $|S'|=r-2$, and such that $S$ is the only $(r-1)$-subset of $S\cup\{w\}$ that is in $B$. Let us prove that $|\{w(u)\colon u\in U\}|\geq(1+o(1))\log_{\alpha}n$. This will immediately imply the desired upper bound.

To that end, it suffices to show that \textbf{whp} for any $(r-1)$-set $S$, and $W$ of size at most $\log_{\alpha} n-5\log_{\alpha}\log_{\alpha}n$, there are at least $2\log_{\alpha}^4n$ vertices $u\in V(G)\setminus (S\cup W)$ such that $S\cup\{u\}\in E(G)$, and for every $w\in W$ there is some $X\subset S\cup\{u\}$ with $|X|=r-1, X\neq S$, such that $X\cup \{w\}\notin E(G)$. Fix $S$ and $W$. The probability that $u$ satisfies the above is $p'=p(1-p^{r-1})^{|W|}\ge p\cdot\frac{\ln^5n}{n}$. These events are independent for different $u$, so the number of vertices satisfying this property is distributed as $Bin(n-(r-1)-|W|,p')$. Thus, by a standard Chernoff's bound, the probability that there are less than $2\log^4_{\alpha}n$ such vertices is at most $\exp\left(-\Omega(\ln^5n)\right)$. There are $\binom{n}{r-1}$ ways to choose $S$ and $\sum_{i=1}^{\log_{\alpha}n}\binom{n}{i}\le n^{\log_{\alpha}n}$ ways to choose $W$, and thus by the union bound there are \textbf{whp} at least $2\log_{\alpha}^4n$ such vertices.
\end{proof}

We now turn to the upper bound of Theorem \ref{th: wsat}\ref{i: strong}. As we follow here the proof strategy from \cite{KS17} with minor adjustments, in the next section we give an outline of the proof. For the full proof, see the Appendix.

\subsection{Outline of the upper bound of Theorem \ref{th: wsat}\ref{i: strong}}
We will utilise two lemmas. The first one generalises a result of Krivelevich \cite{K95} to the case of hypergraphs:
\begin{lemma}
\label{lm:good_subhypergraph - outline}
    Let $r \ge 3$ and $t \ge r$ be integers. There exist positive constants $c_0$ and $c_1$ such that if 
    \[
        \rho \ge c_1 n^{-\frac{t + 1 - r}{\binom{t + 1}{r} - 1}}\quad \text{and}\quad
        k \geq c_0 n^{\frac{\binom{t}{r}(t + 1 - r)}{\left(\binom{t+1}{r} - 1\right)(t-1)}} \log^{\frac{1}{t-1}} n,
    \]    
    then \textbf{whp} $G^{r}(n,\rho)$ contains a subhypergraph $H$ on $[n]$ such that
\begin{itemize}
    \item $H$ is $K^r_{t+1}$-free,
    \item every induced $k$-subhypergraph of $H$ contains a copy of $K^r_t$.
\end{itemize}    
\end{lemma}
The proof of Lemma \ref{lm:good_subhypergraph - outline} follows, in general, ideas present in \cite{K95}, utilising Janson's inequality.

We also require the following lemma, which shows that for $\rho$ chosen as in that in Lemma \ref{lm:good_subhypergraph - outline}, typically every two vertices $u$ and $v$ are not both contained in `many' copies of $K_t^r$ in $G^r(n,\rho)$. This lemma is a new ingredient in the proof --- this was not needed in the graph setting of \cite{KS17}, but is necessary in the hypergraph setting.
\begin{lemma}\label{lemma:max-number-of-copies - outline}
    Let $r \ge 3$ and $t \ge 4$ satisfying $t\geq r$ be integers. Let $c > 0$ be a constant and let $p = c n^{-(t + 1 - r)/\left(\binom{t + 1}{r} - 1\right)}$. Then, \textbf{whp}, for every two vertices $u$ and $v$, the number of copies of $K^r_t$ in $G^{r}(n,\rho)$, which contains $u$ and $v$, is at most $2\binom{n}{t - 2} \rho^{\binom{t}{r}}$.
\end{lemma}
\begin{proof}
    Fix two vertices $u$ and $v$. Denote by $X$ the number of copies of $K_t^r$ in $G \sim G^{r}(n,\rho)$ which contain $u$ and $v$. By the union bound over all the choices of $u$ and $v$, it suffices to prove that
    \[
        \mathbb{P}\left(X > 2\binom{n}{t-2} p^{\binom{t}{r}}\right) \ll n^{-2}.
    \]    

    Set $L = K \log n$ where $K$ is a sufficiently large constant. For every subset $T \subset V(G)$ of size $t$, denote by $Y_T$ the indicator random variable of the event $G[T] \cong K_t^r$. Below, we consider only those $t$-subsets of $V(G)$ that contain $u$ and $v$ (in particular, sum up over such sets in the computation of the $L$-th moment of $X$ below). We have
    \begin{align}
        \mathbb{E}[X^L] &= \sum_{\substack{T_1, \ldots, T_{L} \subset V(G) \\ |T_1| = \ldots = |T_{L}| = t}} \mathbb{E}[Y_{T_1} \cdot \ldots \cdot Y_{T_{L}}]\notag \\
        &= \sum_{\substack{T_1, \ldots, T_{L-1} \subset V(G) \\ |T_1| = \ldots = |T_{L-1}| = t}} \mathbb{E}[Y_{T_1} \cdot \ldots \cdot Y_{T_{L-1}}] \sum_{T_L \subset V(G), |T_L| = t} \mathbb{E}[Y_{T_L} \mid Y_{T_1} = \ldots = Y_{T_{L-1}} = 1]\notag \\
        &= \sum_{\substack{T_1, \ldots, T_{L-1} \subset V(G) \\ |T_1| = \ldots = |T_{L-1}| = t}} \mathbb{E}[Y_{T_1} \cdot \ldots \cdot Y_{T_{L-1}}] \cdot \mathbb{E}[X \mid Y_{T_1} = \ldots = Y_{T_{L-1}} = 1].
    \label{eq:Lth_moment}
    \end{align}

    Let $H$ be a hypergraph with $e(H) \le L \cdot t$. By FKG inequality, $\mathbb{E}[X \mid H \subset G]\geq\mathbb{E}X$.  On the other hand,
    \begin{align*}
        \mathbb{E}[X \mid H \subset G] \le \binom{n}{t-2} \rho^{\binom{t}{r}} + \sum_{j = r}^{t} e(H)^{\binom{j}{r}} n^{t - j} \rho^{\binom{t}{r} - \binom{j}{r}}.
    \end{align*}
    The first summand in the right-hand side above is an upper bound to the expectation of the number of copies of $K_t^r$ which does not share any edges with $H$ while the second summand is an upper bound to the expectation of the number of copies of $K_t^r$ which does share at least one edge with $H$. The index $j$ indicates how many vertices in the copy of $K_t^r$ belongs to $H$. So we have at most $\binom{j}{r}$ edges which are already in $H$. The term $e(H)^{\binom{j}{r}}$ is an upper bound to the number of choices of the common $j$ vertices, $n^{t-j}$ is an upper bound to the number of choices of the remaining $t - j$ vertices and $p^{\binom{t}{r} - \binom{j}{r}}$ is an upper bound to the probability that this copy appears in $G$ conditioned on $H \subseteq G$. We will show that the second summand on the right-hand side of the inequality is asymptotically smaller than 
    \begin{equation}
    \mathbb{E}[X]=(1+o(1))\binom{n}{t-2}\rho^{\binom{t}{r}},
    \label{eq:expect_rooted_cliques}
    \end{equation}
    implying $\mathbb{E}[X\mid H\subset G]\leq(1+o(1))\mathbb{E}[X]$. Note that, for every $j = r, r+1, \ldots, t$, we have
    \begin{align*}
        \frac{n^{t - j} \rho^{\binom{t}{r} - \binom{j}{r}}}{n^{t-2} \rho^{\binom{t}{r}}} = n^{2 - j} \rho^{-\binom{j}{r}} = \Theta\left(n^{2 - j + \frac{\binom{j}{r}(t+1-r)}{\binom{t+1}{r}-1}}\right).
    \end{align*}
    Thus, it is sufficient to show that the last expression tends to 0 as $n\to\infty$, which is true if
    \[
        \frac{\binom{j}{r}(t+1-r)}{\binom{t+1}{r}-1} < j - 2.
    \]

    To show the latter, we will use the following claim.

    \begin{claim}\label{claim:t-condition1}
        $\frac{\binom{t}{r} (t + 1 - r)}{\left(\binom{t+1}{r} - 1\right)(t-1)} < 1 - \frac{r - 1}{t}$.     
    \end{claim}
    \begin{proof}
        We have
        \begin{align*}
            \frac{\binom{t}{r} (t + 1 - r)}{\left(\binom{t+1}{r} - 1\right)(t-1)} &< \frac{\binom{t}{r}}{\binom{t+1}{r} - 1} = \frac{\binom{t}{r}}{\binom{t}{r} + \binom{t}{r - 1} - 1} \\ &= \frac{\binom{t}{r}}{\binom{t}{r} + \binom{t}{r} \cdot \frac{r}{t-r+1}  - 1} = \frac{1}{1 + \frac{r}{t-r+1} - \frac{1}{\binom{t}{r}}}\\
            &\le 1 - \frac{r}{t-r+1} + \frac{1}{\binom{t}{r}} \leq 1 - \frac{r}{t - r + 1} + \frac{1}{t - r + 1} \\
            &=1-\frac{r-1}{t-r+1}< 1 - \frac{r - 1}{t}.
        \end{align*}        
    \end{proof}
    
    By Claim \ref{claim:t-condition1},
    \begin{align*}
        \frac{\binom{j}{r}(t+1-r)}{\binom{t+1}{r}-1} &< \frac{\binom{j}{r}}{\binom{t}{r}} \left(1 - \frac{r - 1}{t}\right)(t-1) \le \frac{\binom{j}{r}}{\binom{t}{r}} (t-2)=
        \frac{j(j-1)\ldots(j-r+1)}{t(t-1)\ldots(t-r+1)}(t-2) \\&\stackrel{r\geq 3} \leq j-2,
    \end{align*}
    as needed.
     
    Thus, we conclude that $\mathbb{E}[X\mid H\subset G]=(1+o(1))\mathbb{E}[X]$ uniformly over all $H$ with $e(H)\leq L\cdot t$.
    By induction, from~\eqref{eq:Lth_moment}, we have
    \[
        \mathbb{E}[X^L] = \biggl((1 + o(1))\mathbb{E}[X]\biggr)^L.
    \]
    From~\eqref{eq:expect_rooted_cliques}, by Markov's inequality, we get
    \begin{align*}
        \mathbb{P}\left(X > 2\binom{n}{t-2} \rho^{\binom{t}{r}}\right) \le \frac{\mathbb{E}[X^L]}{((2+o(1))\mathbb{E}[X])^L} = \left(\frac{(1 + o(1))\mathbb{E}[X]}{2\mathbb{E}[X]}\right)^L = (1/2+o(1))^L \ll n^{-2},
    \end{align*}
    where the last asymptotical inequality is true if $K$ is sufficiently large.
\end{proof}

With these two lemmas at hand, we can now describe our construction. We say that an edge $e\in E(G)\setminus E(H)$ can be \textit{completed} if it closes a copy of $K_s^r$ in $H\cup\{e\}$. Recall that we aim to construct a $K_s^r$-free subhypergraph $H\subseteq G$ with $(1+o(1))\binom{n}{r-1}\log_{\frac{1}{1-p^{r-1}}}n$ edges, such that every edge $e\in E(G)\setminus E(H)$ can be completed. Throughout the proof, we make use of the following observation, already appearing in \cite{KS17}.
\begin{observation}\label{obs: small amount of edges}
It suffices to construct a $K_s^r$-free graph that completes all but $o(n^{r-1}\log n)$ edges. Indeed, by adding each of the uncompleted edges to $H$, if necessary, we obtain a $K_s^r$-saturated graph with an asymptotically equal number of edges.
\end{observation}

We set $\alpha=(1-p^{r-1})^{-1}$ and $\beta=(1-p^{\binom{s}{r}-\binom{s-r}{r}-1})^{-1}$. Throughout the rest of this section, unless explicitly stated otherwise, the base of the logarithms is $\alpha$. We then set aside three disjoint subsets of $[n]$, $A_1$, $A_2$, and $A_3$, of sizes $a_1\coloneqq \frac{1}{p}\log n\left(1+\frac{3}{\log\log n}\right)$, $a_2\coloneqq sr\log_{\beta}n$, and $a_3=\frac{a_2}{\log^{1/r}a_2}$, respectively. We let $B=[n]\setminus (A_1\cup A_2\cup A_3)$, shorthand $t=s-r$, and recall that $G\sim G^r(n,p)$.

We now define $H$ to be a subhypergraph of $G$ with the following edges:
\begin{itemize}
    \item all edges of $G$ intersecting both $A_1$ and $B$ with at most $t$ vertices from $A_1$; and,
    \item if $t\ge r$, we take $H[A_1]\subsetneq G[A_1]$ to be $K_{t+1}^r$-free, such that there exists a copy of $K_t^r$ in every induced subhypergraph of $H[A_1]$ of size at least $c_0a_1^{\frac{\binom{t}{r}(t+1-r)}{\left(\binom{t+1}{r}-1\right)(t-1)}}\log^{1/(t-1)}a_1$. Moreover, every two vertices in $A_1$ are contained in at most $2\binom{a_1}{t-2}\left(c_1a_1\right)^{-\frac{t+1-r}{\binom{t+1}{r}-1}}$ copies of $K_t^r$ in $H[A_1]$. Otherwise, if $t<r$, we take $H[A_1]$ to be the empty hypergraph.
\end{itemize}
Note that, \textbf{whp}, by Lemmas \ref{lm:good_subhypergraph - outline} and \ref{lemma:max-number-of-copies - outline} the subhypergraph described in the second bullet exists. Furthermore, the number of edges we have in $H$ now is indeed $(1+o(1))\binom{n}{r-1}\log_{\frac{1}{1-p^{r-1}}}n$. Moreover, since $H[A_1]$ is $K_{t+1}^r$-free, if $t+1\ge r$ we see that $H$ is $K_s^r$-free. If $t+1\le r-1$, we note that we do not have edges intersecting both $A_1$ and $B$ with at least $t+1$ vertices in $A_1$, and thus $H$ is $K_s^r$-free.

Given an $(r-1)$-set of vertices $S\subseteq [n]$, the neighbourhood of $S$ in $X\subseteq [n]$ is given by $N_X(S)\coloneqq \left\{v\in X\colon\{v\}\cup S\in E(G)\right\}$. Given an $(r-1)$-set of vertices $S\subseteq B$, we say that $S$ is \textit{good} if $|N_{A_1}(S)|\ge pa_1-\frac{\log n}{\log\log n}$, and otherwise we say that $S$ is \textit{bad}. Finally, we say that an edge $e\in G[B]$ is good if there exists at least one good $(r-1)$-subset $S\subseteq e$, and otherwise say that the edge $e$ is \textit{bad}. 

We first utilise the subhypergraph in $H[A_1]$ in order to activate almost all the good edges $e\in E(G[B])$. Observe that good edges $e$ have some $S\subseteq e$, with $|S|=r-1$ and such that $S$ has a large neighbourhood in $A_1$. Fix an $r$-subset $e\subseteq B$. Note that given $S\subseteq e$ with $|S|=r-1$ and its neighbourhood in $A_1$, we have that 
\begin{align*}
    N_{S,e}\coloneqq \left|\left\{v\in N_{A_1}(S)\colon \{v\}\cup S'\in E(G)\text{ for all }S'\subseteq e, |S'|=r-1\right\}\right|\sim Bin(|N_{A_1}(S)|,p^{r-1}).
\end{align*}
By the choice of $H[A_1]$, we can then expose the edges of the form $S'\cup \{v\}$ in $G$, where $S'\subseteq e$ and $v\in N_{A_1}(S)$, and then using Chernoff's bound find in $N_{S,e}$ many copies of $K_t^r$. Now, note that given two copies of $K_t^r$ in $H[N_{S,e}]$ the events that each of them closes a copy of $K_s^r$ with $e$ are dependent only if these copies intersect in two vertices. To bound these dependencies, we utilise the fact that in $H[A_1]$ every two vertices do not have many copies of $K_t^r$ sharing them. This, together with Janson's inequality, allows us to show that \textbf{whp}, for all but $o(n^{r-1}\log n)$ good edges in $e\in E(G[B])$, there is at least one copy of $K_t^r$ that closes a copy of $K_s^r$ with $e$. Note that by Observation \ref{obs: small amount of edges}, this in fact completes our treatment for all the good edges in $G[B]$.

Before turning to bad edges in $G[B]$, let us discuss the difference in the treatment of good edges above from the setting of graphs in \cite{KS17}. Indeed, in the case of graphs, a key part of the argument for the upper bound is to consider vertices $v$ with a large neighbourhood in $A_1$. Then, fixing $v$ and considering edges $uv\in E(G[B]))$, it suffices to find a copy of $K_{s-2}^2$ in $N_v \cap N_u$ -- whose size is distributed according to $Bin(|N_v|,p)$. On the other hand, when considering hypergraphs, we consider instead an $(r-1)$-set $S\subseteq e$ with a substantial neighbourhood in $A_1$, $N_{S,e}$. Then, when considering edges $\{u\}\cup S\in E(G[B])$, it no longer suffices to find a copy of $K_{s-r}^r$ in $\Tilde{N}_{S, u} = \bigcap_{S' \subset S \cup \{u\}, |S'| = r - 1} N_{S',e}$. Indeed, one needs to further consider edges with $1\le \ell < r-1$ vertices from $S \cup \{u\}$, and $r-\ell$ vertices from $\Tilde{N}_{S, u}$. As noted above, this further creates possible dependencies which do not appear in the case of graphs, and, in particular, this is why Lemma \ref{lemma:max-number-of-copies - outline} is important to us in the hypergraph setting.

We now turn to bad edges in $G[B]$. As there is a non-negligible portion of them, we require the set $A_2$ to complete them. By Lemma \ref{lm:good_subhypergraph - outline} there exists a subhypergraph $H_2\subseteq G[A_2]$ which is $K_{t+1}^r$-free and every large enough subset in it induces a copy of $K_t^r$, when $t\ge r$. Moreover, there are at least $r\log_{\beta}n$ vertex-disjoint copies of $K_t^r$ in $H_2[A_2]$. We can thus add to $H$ the edges of $H_2[A_2]$, edges intersecting both $A_2$ and $B$ in $G$ with at least two and at most $t$ vertices from $A_2$, and edges of the form $S\cup \{v\}$ where $S\subseteq B$ is bad and $v\in A_2$. The first two are of order $o(n^{r-1}\log n)$, and using Chernoff's bound, one can show that there are at most $\frac{n^{r-1}}{\log n}$ bad $(r-1)$-subsets $S\subseteq B$, and thus the last set of edges is also of order $o(n^{r-1}\log n)$. Similarly to before, $H$ is still $K_s^r$-free. Using Markov's inequality, \textbf{whp} we can now close all bad edges in $G[B]$ (recall that a bad edge has that all its $(r-1)$-subsets are bad).

Note that we have not added any edge to $H$ intersecting both $B$ and $A_2$, whose intersection with $B$ is a good $(r-1)$-subset. As there are $\Omega(n^{r-1})$ good $(r-1)$-subsets in $B$, and $\Omega(\log n)$ vertices in $A_2$, we cannot ignore these edges, and this is the reason for setting aside the set $A_3$. Once again, we find in $G[A_3]$ a subhypergraph $H_3$ which is $K_{t+1}^r$-free, and every large enough subset in it induces a copy of $K_t^r$, when $t\ge r$. 

The choice of edges which we add to $H$ now is slightly more delicate. We add to $H$ the edges of $H_3$, all edges intersecting both $B$ and $A_3$ with at least two and at most $t$ vertices from $A_3$, edges of the form $S\cup \{v\}$ where $S\subseteq B$ is good and $v\in A_3$, and edges intersecting both $B\cup A_2$ and $A_3$ with one vertex from $A_2$ and at least one and at most $t$ vertices from $A_3$. This careful choice of edges allows us to show that the hypergraph $H$ is still $K_s^r$-free. As the argument here is slightly more delicate, let us state it explicitly. Since there are no edges intersecting both $A_1$ and $A_2\cup A_3$, it suffices to consider $X\subseteq B\cup A_{2}\cup A_{3}$ of size $s$, and suppose towards contradiction that $X$ induces a clique in $H$. Since $H[B]$ is an empty hypergraph, we must have that $|X\cap B|\le r-1$. Since we only added edges intersecting $A_3$, we may assume that $X$ contains vertices from $A_3$. If $|X\cap A_3|\ge t+1$, we are done by similar arguments to before. Furthermore, if there are at least two vertices in $A_2\cap X$, note that there are no edges in $H$ containing two vertices from $A_2$ and at least one vertex from $A_3$, leading to a contradiction. Thus, we are left with the case where $|X\cap A_{2}|=1, |X\cap A_{3}|=t, |X\cap B|=r-1$. Now, if $X\cap B$ is bad, there are no edges of the form $(X\cap B)\cup \{v\}$, $\{v\}\in A_3$, in $H$, and if it is good, then there are no edges of the form $(X\cap B)\cup \{v\}$, $\{v\}\in A_2$, in $H$ --- either way, we get that $X$ cannot induce a clique. Let us note that this last part of the argument is also quite different from the argument in the setting of graphs. Indeed, in \cite{KS17}, to ensure that $H$ remains $K_s^2$-free, one needs to consider a partition of $A_2$ into independent sets, and partition $A_3$ accordingly. Here, adding edges intersecting both $B\cup A_2$ and $A_3$ which intersect $A_2$ in at most one vertex suffices to show that $H$ remains $K_s^r$-free.

Furthermore, due to the size of $A_3$, the number of edges we added at this stage is of order $o(n^{r-1}\log n)$. Using similar arguments to before, we can show that we can activate all but $o(n^{r-1}\log n)$ of the edges of the form $S\cup \{v\}$ where $S\subseteq B$ is good and $v\in A_2$. Again, by Observation \ref{obs: small amount of edges}, this completes our treatment for all edges of the form $S\cup\{v\}$ where $S\subseteq B$ is good and $v\in A_2$.

Finally, note that we have not completed all the edges: we did not complete edges of the form $S\cup \{v\}$ where $S\subseteq B$ is bad and $v\in A_3$ as well as some of the edges induced by $A_1\cup A_2\cup A_3$. Furthermore, if $t<r$, we have not completed edges instersecting both $B$ and $A_1\cup A_2\cup A_3$ with more than $t$ vertices in either $A_1$, $A_2$, or $A_3$. However, there are at most $o(n^{r-1}\log n)$ edges of these types, and we are thus done by Observation \ref{obs: small amount of edges}.

\bibliographystyle{abbrv}
\bibliography{sat}

\begin{thebibliography}{10}

\bibitem{A85}
N.~Alon.
\newblock An extremal problem for sets with applications to graph theory.
\newblock {\em J. Combin. Theory Ser. A}, 40(1):82--89, 1985.

\bibitem{AS16}
N.~{Alon} and J.~H. {Spencer}.
\newblock {\em {The probabilistic method}}.
\newblock Hoboken, NJ: John Wiley \& Sons, fourth edition, 2016.

\bibitem{BMTZ20}
M.~Bidgoli, A.~Mohammadian, B.~Tayfeh-Rezaie, and M.~Zhukovskii.
\newblock Threshold for stability of weak saturation.
\newblock {\em J. Graph Theory}, 106(3):474--495, 2024.

\bibitem{B65}
B.~Bollob\'{a}s.
\newblock On generalized graphs.
\newblock {\em Acta Math. Acad. Sci. Hungar.}, 16:447--452, 1965.

\bibitem{B68}
B.~Bollob\'{a}s.
\newblock Weakly {$k$}-saturated graphs.
\newblock In {\em Beitr\"{a}ge zur {G}raphentheorie ({K}olloquium, {M}anebach,
  1967)}, pages 25--31. B. G. Teubner Verlagsgesellschaft, Leipzig, 1968.

\bibitem{CycleSat}
Y.~Demidovich, A.~Skorkin, and M.~Zhukovskii.
\newblock Cycle saturation in random graphs.
\newblock {\em SIAM J. Discrete Math.}, 37(3):1359--1385, 2023.

\bibitem{DZ22}
S.~Demyanov and M.~Zhukovskii.
\newblock Tight concentration of star saturation number in random graphs.
\newblock {\em Discrete Math.}, 346(10):113572, 2023.

\bibitem{DHZ23}
S.~Diskin, I.~Hoshen, and M.~Zhukovskii.
\newblock A jump of the saturation number in random graphs?
\newblock {\em Random Structures Algorithms}, 66(4):Paper No. e70009, 20, 2025.

\bibitem{EHM64}
P.~Erd\H{o}s, A.~Hajnal, and J.~W. Moon.
\newblock A problem in graph theory.
\newblock {\em Amer. Math. Monthly}, 71:1107--1110, 1964.

\bibitem{F82}
P.~Frankl.
\newblock An extremal problem for two families of sets.
\newblock {\em European J. Combin.}, 3(2):125--127, 1982.

\bibitem{janson1990exponential}
S.~Janson, T.~Luczak, and A.~Rucinski.
\newblock An exponential bound for the probability of nonexistence of a
  specified subgraph in a random graph.
\newblock In {\em Random graphs}, volume~87, pages 73--87, 1990.

\bibitem{K84}
G.~Kalai.
\newblock Weakly saturated graphs are rigid.
\newblock In {\em Convexity and graph theory ({J}erusalem, 1981)}, volume~87 of
  {\em North-Holland Math. Stud.}, pages 189--190. North-Holland, Amsterdam,
  1984.

\bibitem{K85}
G.~Kalai.
\newblock Hyperconnectivity of graphs.
\newblock {\em Graphs Combin.}, 1(1):65--79, 1985.

\bibitem{KMMT23}
O.~Kalinichenko, M.~Miralaei, A.~Mohammadian, and B.~Tayfeh-Rezaie.
\newblock Weak saturation numbers in random graphs.
\newblock {\em arXiv preprint arXiv:2306.10375}, 2023.

\bibitem{KZ23}
O.~Kalinichenko and M.~Zhukovskii.
\newblock Weak saturation stability.
\newblock {\em European J. Combin.}, 114:Paper No. 103777, 21, 2023.

\bibitem{KS17}
D.~Kor\'{a}ndi and B.~Sudakov.
\newblock Saturation in random graphs.
\newblock {\em Random Structures Algorithms}, 51(1):169--181, 2017.

\bibitem{K95}
M.~Krivelevich.
\newblock Bounding {R}amsey numbers through large deviation inequalities.
\newblock {\em Random Structures Algorithms}, 7(2):145--155, 1995.

\bibitem{L77}
L.~Lov\'{a}sz.
\newblock Flats in matroids and geometric graphs.
\newblock In {\em Combinatorial surveys ({P}roc. {S}ixth {B}ritish
  {C}ombinatorial {C}onf., {R}oyal {H}olloway {C}oll., {E}gham, 1977)}, pages
  45--86. Academic Press, London-New York, 1977.

\bibitem{MT18}
A.~Mohammadian and B.~Tayfeh-Rezaie.
\newblock Star saturation number of random graphs.
\newblock {\em Discrete Math.}, 341(4):1166--1170, 2018.

\bibitem{Z49}
A.~A. Zykov.
\newblock On some properties of linear complexes.
\newblock {\em Mat. Sbornik N.S.}, 24(66):163--188, 1949.

\end{thebibliography}

\appendix
\section{Proof of the upper bound of Theorem \ref{th: wsat}\ref{i: strong}}
\subsection{Preliminaries}\label{section:prem}
Let us first state (several versions of) Janson's inequality, which we will use later on. In some cases, we will need to bound the probability of the existence of many edge-disjoint copies of some graph in $G(n,p)$. Janson's inequality is a well-known tool, frequently utilised in such problems.  

Let $\Omega$ be a finite universal set and let $R$ be a random subset of $\Omega$ containing every element $r \in \Omega$ with probability $p_r$, independently. Let $I$ be a finite set of indices, and let $\{A_i\}_{i \in I}$ be subsets of $\Omega$. For every $i \in I$, let $X_i$ be the indicator random variable of the event that $B_i \subset R$ and let $X = \sum_{i \in I} X_i$. For $i, j \in I$, we write $i \sim j$ if $A_i \cap A_j \neq \varnothing$. Define
\begin{align*}
    \mu &\coloneqq \mathbb{E}[X] = \sum_{i \in I} \Pr(X_i = 1) \quad \text{and}\\
    \Delta &\coloneqq \sum_{i \sim j} \Pr(X_i = X_j = 1),
\end{align*}
where the sum in $\Delta$ goes over all ordered pairs $i, j \in I$ such that $i \sim j$.

\begin{thm}[{Janson's inequality~\cite{janson1990exponential}}]\label{theorem:janson}
  $\Pr(X = 0) \le e^{-\mu^2 / 2\Delta}$.
\end{thm}

We also use versions of Janson's inequality (see, for example, \cite{AS16}) which bound the probability of the event that we do not have many disjoint copies of $A_i$'s in $R$. We call a set of indices $J \subseteq I$ a \textit{disjoint family} if for every $i \neq j \in J$ we have $A_i \cap A_j = \varnothing$. Moreover, we call a set of indices $J \subseteq I$ a \textit{maximal disjoint family} if $J$ is a disjoint family and there is no $i \in I \setminus J$ such that $A_i \cap A_j = \varnothing$ for every $j \in J$. For every $s$, denote by $D_s$ and $M_s$ the events that there exists a disjoint family of size $s$ and a maximal disjoint family of size $s$, respectively.

\begin{lemma}\label{lemma:janson-disjoint-family}
    For every integer $s$, 
    $$
        \Pr(D_s) \le \frac{\mu^s}{s!}.
    $$
\end{lemma}

For a probability bound on the size of a maximal disjoint family, we need another notation. Define
\begin{align*}
    \nu = \max_{j \in I} \sum_{i \sim j} \Pr(X_i = 1).
\end{align*}

\begin{lemma}\label{lemma:janson-max-disjoint}
    For every integer $s$, 
    $$
        \Pr(M_s) \le \frac{\mu^s}{s!} e^{-\mu + s\nu + \Delta/2}.    
    $$
\end{lemma}

Let us derive from Lemma \ref{lemma:janson-max-disjoint} an upper bound to the probability that there exists a maximal disjoint family of size $s \le (1 - \epsilon)\mu$ for some $\epsilon \in (0, 1)$. 

\begin{corollary}\label{corollary:janson-max-disjoint}
    If $\Delta = o(\mu)$ and $\nu = o(1)$, then, for every $\epsilon\in (0,1)$,
    $$
        \Pr\left(\bigcup_{s \le (1-\epsilon)\mu}\ M_s\right) \le e^{-\epsilon^2 \mu / 3}.
    $$
\end{corollary}
\begin{proof}
    By the union bound, 
    \begin{align*}
        \Pr\left(\bigcup_{s \le (1-\epsilon)\mu}\ M_s\right) &\le \sum_{s=1}^{(1-\epsilon)\mu} \Pr(M_s) \le \sum_{s=1}^{(1-\epsilon)\mu} \frac{\mu^s}{s!} e^{-\mu} e^{s\nu + \Delta/2} \\
        &\le e^{\epsilon^3 \mu} \sum_{s=1}^{(1-\epsilon)\mu} \frac{\mu^s}{s!} e^{-\mu},
    \end{align*}
    where the second inequality is true by Lemma \ref{lemma:janson-max-disjoint} and the last inequality is true since, by the assumptions of the corollary, $s\nu + \Delta/2 \le \epsilon^2 \mu / 6$. Note that the sum above equals to the probability that the value of a Poisson random variable with mean $\mu$ is at most $(1-\epsilon)\mu$. Thus, by a typical bound on the tails of the Poisson distribution (see, for example, \cite[Theorem A.1.15]{AS16}),
    \begin{align*}
        \Pr\left(\bigcup_{s \le (1-\epsilon)\mu}\ M_s\right) \le e^{\epsilon^2 \mu / 6}e^{-\epsilon^2 \mu / 2} = e^{-\epsilon^2 \mu / 3}.
    \end{align*}
\end{proof}

We also make use of the following bounds (see, for example, \cite[Claim 2.1]{KS17}).
\begin{claim}\label{claim:binomial-bounds}
    Let $0 < p < 1$ be a constant and $X \sim \text{Bin}(n, p)$ be a binomial variable. Then, for sufficiently large $n$ and for any $a>0$,
    \begin{enumerate}
        \item $\Pr(X \ge np + a) \le e^{-\frac{a^2}{2(np + a/3)}}$,
        \item $\Pr(X \le np - a) \le e^{-\frac{a^2}{2np}}$ and
        \item $\Pr\left(X \le \frac{n}{\log^2 n}\right) \le (1-p)^{n - \frac{n}{\log n}}$.
    \end{enumerate}
\end{claim}

\subsection{Properties of random hypergraphs}\label{sec: properties}
In this section, we collect some properties of the random hypergraph that will be useful to us throughout the proof.
\begin{lemma}\label{lemma:max-number-of-copies}
    Let $r \ge 3$ and $t \ge 4$ satisfying $t\geq r$ be integers. Let $c > 0$ be a constant and let $p = c n^{-(t + 1 - r)/\left(\binom{t + 1}{r} - 1\right)}$. Then, \textbf{whp}, for every two vertices $u$ and $v$, the number of copies of $K_t^r$ in $G^{r}(n,p)$, which contains $u$ and $v$, is at most $2\binom{n}{t - 2} p^{\binom{t}{r}}$.
\end{lemma}
\begin{proof}
    Fix two vertices $u$ and $v$. Denote by $X$ the number of copies of $K_t^r$ in $G \sim G^{r}(n,p)$ which contain $u$ and $v$. By the union bound over all the choices of $u$ and $v$, it suffices to prove that
    \[
        \mathbb{P}\left(X > 2\binom{n}{t-2} p^{\binom{t}{r}}\right) \ll n^{-2}.
    \]    

    Set $L = K \log n$ where $K$ is a sufficiently large constant. For every subset $T \subset V(G)$ of size $t$, denote by $Y_T$ the indicator random variable of the event $G[T] \cong K_t^r$. Below, we consider only those $t$-subsets of $V(G)$ that contain $u$ and $v$ (in particular, sum up over such sets in the computation of the $L$-th moment of $X$ below). We have
    \begin{align}
        \mathbb{E}[X^L] &= \sum_{\substack{T_1, \ldots, T_{L} \subset V(G) \\ |T_1| = \ldots = |T_{L}| = t}} \mathbb{E}[Y_{T_1} \cdot \ldots \cdot Y_{T_{L}}]\notag \\
        &= \sum_{\substack{T_1, \ldots, T_{L-1} \subset V(G) \\ |T_1| = \ldots = |T_{L-1}| = t}} \mathbb{E}[Y_{T_1} \cdot \ldots \cdot Y_{T_{L-1}}] \sum_{T_L \subset V(G), |T_L| = t} \mathbb{E}[Y_{T_L} \mid Y_{T_1} = \ldots = Y_{T_{L-1}} = 1]\notag \\
        &= \sum_{\substack{T_1, \ldots, T_{L-1} \subset V(G) \\ |T_1| = \ldots = |T_{L-1}| = t}} \mathbb{E}[Y_{T_1} \cdot \ldots \cdot Y_{T_{L-1}}] \cdot \mathbb{E}[X \mid Y_{T_1} = \ldots = Y_{T_{L-1}} = 1].
    \label{eq:Lth_moment_ap}
    \end{align}

    Let $H$ be a hypergraph with $e(H) \le L \cdot t$. By FKG inequality, $\mathbb{E}[X \mid H \subset G]\geq\mathbb{E}X$.  On the other hand,
    \begin{align*}
        \mathbb{E}[X \mid H \subset G] \le \binom{n}{t-2} p^{\binom{t}{r}} + \sum_{j = r}^{t} e(H)^{\binom{j}{r}} n^{t - j} p^{\binom{t}{r} - \binom{j}{r}}.
    \end{align*}
    The first summand in the right-hand side above is an upper bound to the expectation of the number of copies of $K_t^r$ which does not share any edges with $H$ while the second summand is an upper bound to the expectation of the number of copies of $K_t^r$ which does share at least one edge with $H$. The index $j$ indicates how many vertices in the copy of $K_t^r$ belongs to $H$. So we have at most $\binom{j}{r}$ edges which are already in $H$. The term $e(H)^{\binom{j}{r}}$ is an upper bound to the number of choices of the common $j$ vertices, $n^{t-j}$ is an upper bound to the number of choices of the remaining $t - j$ vertices and $p^{\binom{t}{r} - \binom{j}{r}}$ is an upper bound to the probability that this copy appears in $G$ conditioned on $H \subseteq G$. We will show that the second summand on the right-hand side of the inequality is asymptotically smaller than 
    \begin{equation}
    \mathbb{E}[X]=(1+o(1))\binom{n}{t-2}p^{\binom{t}{r}},
    \label{eq:expect_rooted_cliques_ap}
    \end{equation}
    implying $\mathbb{E}[X\mid H\subset G]\leq(1+o(1))\mathbb{E}[X]$. Note that, for every $j = r, r+1, \ldots, t$, we have
    \begin{align*}
        \frac{n^{t - j} p^{\binom{t}{r} - \binom{j}{r}}}{n^{t-2} p^{\binom{t}{r}}} = n^{2 - j} p^{-\binom{j}{r}} = \Theta\left(n^{2 - j + \frac{\binom{j}{r}(t+1-r)}{\binom{t+1}{r}-1}}\right).
    \end{align*}
    Thus, it is sufficient to show that the last expression tends to 0 as $n\to\infty$, which is true if
    \[
        \frac{\binom{j}{r}(t+1-r)}{\binom{t+1}{r}-1} < j - 2.
    \]

    To show the latter, we will use the following claim.

    \begin{claim}\label{claim:t-condition1_ap}
        $\frac{\binom{t}{r} (t + 1 - r)}{\left(\binom{t+1}{r} - 1\right)(t-1)} < 1 - \frac{r - 1}{t}$.     
    \end{claim}
    \begin{proof}
        We have
        \begin{align*}
            \frac{\binom{t}{r} (t + 1 - r)}{\left(\binom{t+1}{r} - 1\right)(t-1)} &< \frac{\binom{t}{r}}{\binom{t+1}{r} - 1} = \frac{\binom{t}{r}}{\binom{t}{r} + \binom{t}{r - 1} - 1} \\ &= \frac{\binom{t}{r}}{\binom{t}{r} + \binom{t}{r} \cdot \frac{r}{t-r+1}  - 1} = \frac{1}{1 + \frac{r}{t-r+1} - \frac{1}{\binom{t}{r}}}\\
            &\le 1 - \frac{r}{t-r+1} + \frac{1}{\binom{t}{r}} \leq 1 - \frac{r}{t - r + 1} + \frac{1}{t - r + 1} \\
            &=1-\frac{r-1}{t-r+1}< 1 - \frac{r - 1}{t}.
        \end{align*}        
    \end{proof}
    
    By Claim \ref{claim:t-condition1_ap},
    \begin{align*}
        \frac{\binom{j}{r}(t+1-r)}{\binom{t+1}{r}-1} &< \frac{\binom{j}{r}}{\binom{t}{r}} \left(1 - \frac{r - 1}{t}\right)(t-1) \le \frac{\binom{j}{r}}{\binom{t}{r}} (t-2)=
        \frac{j(j-1)\ldots(j-r+1)}{t(t-1)\ldots(t-r+1)}(t-2) \\&\stackrel{r\geq 3} \leq j-2,
    \end{align*}
    as needed.
    
    Thus, we conclude that $\mathbb{E}[X\mid H\subset G]=(1+o(1))\mathbb{E}[X]$ uniformly over all $H$ with $e(H)\leq L\cdot t$.
    By induction, from~\eqref{eq:Lth_moment_ap}, we have
    \[
        \mathbb{E}[X^L] = \biggl((1 + o(1))\mathbb{E}[X]\biggr)^L.
    \]
    From~\eqref{eq:expect_rooted_cliques_ap}, by Markov's inequality, we get
    \begin{align*}
        \mathbb{P}\left(X > 2\binom{n}{t-2} p^{\binom{t}{r}}\right) \le \frac{\mathbb{E}[X^L]}{((2+o(1))\mathbb{E}[X])^L} = \left(\frac{(1 + o(1))\mathbb{E}[X]}{2\mathbb{E}[X]}\right)^L = (1/2+o(1))^L \ll n^{-2},
    \end{align*}
    where the last asymptotical inequality is true if $K$ is sufficiently large.
\end{proof}

\begin{lemma}
\label{lm:good_subhypergraph}
    Let $r \ge 3$ and $t \ge r$ be integers. There exist positive constants $c_0$ and $c_1$ such that if 
    \[
        p \ge c_1 n^{-\frac{t + 1 - r}{\binom{t + 1}{r} - 1}}\quad \text{and}\quad
        k \geq c_0 n^{\frac{\binom{t}{r}(t + 1 - r)}{\left(\binom{t+1}{r} - 1\right)(t-1)}} \log^{\frac{1}{t-1}} n,
    \]    
    then \textbf{whp} $G^{r}(n,p)$ contains a subhypergraph $H$ on $[n]$ such that
    \begin{itemize}
    \item $H$ is $K_{t+1}^r$-free,
    \item every induced $k$-subhypergraph of $H$ contains a copy of $K_t^r$.
    \end{itemize}    
\end{lemma}
\begin{proof}
    We may assume that $k$ is the smallest possible integer satisfying the requirement in the statement of Lemma \ref{lm:good_subhypergraph}, that is
    $$
        k = \left\lceil c_0 n^{\frac{\binom{t}{r}(t + 1 - r)}{\left(\binom{t+1}{r} - 1\right)(t-1)}} \log^{\frac{1}{t-1}} n \right\rceil.
    $$
    Also, since the property that the lemma claims is increasing, without the loss of generality, we may set 
    $$
    p = c_1 n^{-\frac{t + 1 - r}{\binom{t + 1}{r} - 1}}.
    $$

    Let $G \sim G^{r}(n,p)$. Let $\mathcal{K}$ be a set of all graphs obtained as a union of $K_t^r$ and $K_{t+1}^r$ sharing at least $r$ vertices (we include a single representative of every isomorphism class into $\mathcal{K}$). 
    For $S \subset [n]$,
    \begin{itemize}
    \item denote by $X_S$ the maximum size of a set of pairwise edge-disjoint copies of $K_t^r$ in $G[S]$;
    \item denote by $Y_S$ the number of subhypergraphs $H\subset G$ isomorphic to a hypergraph from $\mathcal{K}$ and such that $H[V(H)\cap S]$ contains a copy of $K_t^r$;
    \item denote by $Z_S$ the maximum size of a set of pairwise edge-disjoint hypergraphs $H\subset G$ isomorphic to a hypergraph from $\mathcal{K}$ and such that $H[V(H)\cap S]$ contains a copy of $K_t^r$.
    \end{itemize}
    Let $A_S$ be the event that $X_S > \lambda Z_S$ where
    \[
        \lambda := \binom{t}{r}\left(\binom{t+1}{r} - 1\right) + 1.
    \]

    \begin{claim}
        If $A_S$ holds for every subset $S \subset V(G)$ of size $k$, then $G$ contains a subhypergraph $H$ which is $K_{t+1}^r$-free and every induced $k$-subhypergraph of $H$ contains a copy of $K_t^r$.
    \end{claim}
    \begin{proof}
        Let $\mathcal{H}$ be an inclusion-maximal family of edge-disjoint copies of $K_{t+1}^r$ in $G$. Denote by $G' \subset G$ the subhypergraph obtained by deleting from $G$ all edges that appear in a clique from $\mathcal{H}$. Note that $G'$ is $K_{t+1}^r$-free. We claim that every induced $k$-subhypergraph of $G'$ contains $K_t^r$, i.e. $G'$ is the desired hypergraph.
    
        Assume by contradiction that there exists a $k$-set $S \subset [n]$ such that there is no copy of $K_t^r$ in $G'[S]$. Let $\mathcal{K}_S$ be a family of edge-disjoint copies of $K_t^r$ in $G[S]$ of size $X_S$. Consider an auxiliary graph $G_S$ with the vertex set $\mathcal{K}_S$ and the edge set defined as follows: cliques $K^1,K^2$ from $\mathcal{K}_S$ are adjacent if there exists $K'\in\mathcal{H}$ sharing an edge with $K^1$ and sharing an edge with $K^2$. Fix $K\in\mathcal{K}_S$. Since cliques in $\mathcal{H}$ are edge-disjoint, $K$ has common edges with at most $\binom{t}{r}$ elements of $\mathcal{H}$. On the other hand, each clique from $\mathcal{H}$ that has a common edge with $K$ shares edges with at most $\binom{t+1}{r} - 1$ other cliques from $\mathcal{K}_S$. Thus, the maximum degree of $G_S$ is at most $\binom{t}{r}\left(\binom{t+1}{r} - 1\right) = \lambda - 1$. 
        Hence, $G_S$ contains an independent set of size at least $X_S / \lambda$. However, since $X_S / \lambda > Z_S$, this corresponds to a family of strictly more than $Z_S$ edge-disjoint copies of hypergraphs from $\mathcal{K}$ having a copy of $K_t^r$ in $S$. Hence, there exists a copy of $K_t^r$ in $G[S]$ which does not share any edges with elements from $\mathcal{H}$ and thus this copy exists also in $G'[S]$, a contradiction.
    \end{proof}

    To finish the proof of Lemma~\ref{lm:good_subhypergraph} it suffices to prove the following claim.
    \begin{claim}
        \textbf{Whp} $A_S$ holds for every $S \subset V(G)$ of size $k$.
    \end{claim}
    \begin{proof}
        Fix $S \subset V(G)$ of size $k$. We first bound a lower-tail of $X_S$ using Corollary \ref{corollary:janson-max-disjoint}. Next, we bound an upper-tail of $Z_S$ using Lemma \ref{lemma:janson-disjoint-family}. We use the notations of Janson's inequality described in Section \ref{section:prem}.
        
        Let us start with a lower-tail estimate of $X_S$. Denote by $D_S$ the number of copies of $K_t^r$ in $G[S]$. We have
        \begin{align*}
            \mu_D &:= \mathbb{E}[D_S] 
            = \binom{k}{t} p^{\binom{t}{r}}, \\
            \Delta_D &:= \sum_{\substack{T_1, T_2 \subset S \\ |T_1| = |T_2| = t \\ |T_1 \cap T_2| \ge r}}\mathbb{P}(G[T_1]\cong G[T_2]\cong K_t^r)
            = \mu_D \sum_{i=r}^{t - 1} \binom{t}{i} \binom{k - t}{t - i} p^{\binom{t}{r} - \binom{i}{r}}, \\
            \nu_D &= \max_{T \subset S, |T| = t} \sum_{i=r}^{t - 1} \binom{t}{i} \binom{k - t}{t - i} p^{\binom{t}{r}}.
        \end{align*}

        To make sure we can apply Corollary \ref{corollary:janson-max-disjoint}, we show that $\mu_D = o(1)$ and $\Delta_D = o(\mu_D)$. We have        
        \begin{align*}
            \nu_D &= \Theta\left(k^{t - r} p^{\binom{t}{r}}\right) = \Theta\left(n^{\frac{(t-r)\binom{t}{r}(t+1-r)}{\left(\binom{t+1}{r} - 1\right)(t-1)} - \frac{\binom{t}{r}(t+1-r)}{\binom{t + 1}{r} - 1}} \log^{\frac{t-r}{t-1}} n\right) \\
            &= \Theta\left(n^{\frac{\binom{t}{r}(t+1-r)}{\binom{t+1}{r} - 1}\left(\frac{t-r}{t-1} - 1\right)} \log^{\frac{t-r}{t-1}} n\right) = o(1).
        \end{align*}
        In order to show that $\Delta_D = o(\mu_D)$, we need to show that $k^{t-i} p^{\binom{t}{r} - \binom{i}{r}} = o(1)$ for every integer $r \le i \le t - 1$. Fix such an integer $i$. Then,
        \begin{align*}
            k^{t-i} p^{\binom{t}{r} - \binom{i}{r}} &= n^{\frac{(t-i)\binom{t}{r}(t+1-r)}{\left(\binom{t+1}{r} - 1\right)(t-1)} - \left(\binom{t}{r} - \binom{i}{r}\right)\frac{t+1-r}{\binom{t + 1}{r} - 1}} \log^{\frac{t-i}{t-1}} n \\
            &= n^{\frac{t+1-r}{\binom{t+1}{r} - 1}\left(\frac{(t-i)\binom{t}{r}}{t-1} - \binom{t}{r} + \binom{i}{r}\right)} \log^{\frac{t-i}{t-1}} n \\
            &= n^{\frac{t+1-r}{\binom{t+1}{r} - 1}\left(-\frac{i - 1}{t-1}\binom{t}{r} + \binom{i}{r}\right)} \log^{\frac{t-i}{t-1}} n = o(1),
        \end{align*}
        where the last equality is true since $\binom{t}{r} > \binom{i}{r} \cdot \frac{t - 1}{i - 1}$. Indeed,
        \begin{align*}
            \binom{t}{r} > \binom{i}{r} \cdot \frac{t - 1}{i - 1} &\Longleftrightarrow \frac{t!}{(t-r)!} > \frac{i!}{(i-r)!} \cdot \frac{t-1}{i-1} \Longleftrightarrow \frac{t(t-1)\cdot \ldots \cdot(t-r+1)}{i(i-1)\cdot \ldots \cdot(i-r+1)} > \frac{t-1}{i-1},
        \end{align*}
        and the last inequality is true since we have $t > i$ and $r \ge 2$ and thus 
        $$
            \frac{t(t-1)\cdot \ldots \cdot(t-r+1)}{i(i-1)\cdot \ldots \cdot(i-r+1)} \ge \frac{t}{i} \cdot \frac{t-1}{i-1} > \frac{t-1}{i - 1}.
        $$

        Therefore, by Corollary \ref{corollary:janson-max-disjoint}, for every $\epsilon \in (0,1)$,         
        \begin{align}\label{eq:X_S-bound}
            \mathbb{P}(X_S < (1-\epsilon)\mu_D) \le e^{-\epsilon^2 \mu_D/3}.
        \end{align}

        Next, we analyse the upper-tail of $Z_S$. We will denote $\mathbb{E}[Z_S]$ and $\mathbb{E}[Y_S]$ by $\mu_Z$ and $\mu_Y$, respectively. Clearly, $Z_S \le Y_S$ and thus $\mu_Z \le \mu_Y$. Thus, by Lemma \ref{lemma:janson-disjoint-family}, 
        \begin{align}\label{eq:Z_S-bound}
            \mathbb{P}(Z_S \ge 5 \mu_Y) \le \frac{\mu_Z^{5\mu_Y}}{(5 \mu_Y)!} \le \left(\frac{e \mu_Z}{5 \mu_Y}\right)^{5\mu_Y} \le e^{-(\log 5 - 1) 5\mu_Y}.
        \end{align}

        Calculating $\mu_Y$,
        \begin{align*}
            \mu_Y = \Theta\left(k^t \sum_{\ell = r}^{t} n^{t - \ell + 1} p^{\binom{t+1}{r} + \binom{t}{r} - \binom{\ell}{r}}\right).
        \end{align*}
        For every $t \ge \ell > r$,
        \begin{align*}
            &k^t n^{t-r+1} p^{\binom{t+1}{r} + \binom{t}{r} - \binom{r}{r}} \gg k^t n^{t - \ell + 1} p^{\binom{t+1}{r} + \binom{t}{r} - \binom{\ell}{r}} \\&\Longleftrightarrow n^{\ell - r} \gg p^{1-\binom{\ell}{r}} \\
            &\Longleftrightarrow \ell - r > -\left(1-\binom{\ell}{r}\right)\frac{t+1 - r}{\binom{t+1}{r}-1} \\
            &\Longleftrightarrow \frac{\ell - r}{\binom{\ell}{r} - 1} > \frac{t + 1 - r}{\binom{t+1}{r} - 1}.
        \end{align*}
        The last inequality is true since the function $f(x) = \frac{x - r}{\binom{x}{r} - 1}$ is decreasing when $x > r$. Thus, there exist positive constants $c_2$ and $c_3$, which do not depend on $c_0$ and $c_1$, such that
        \[
            c_2 \binom{k}{t} n^{t - r + 1} p^{\binom{t}{r} + \binom{t+1}{r} - 1} \le \mu_Y \le c_3 \binom{k}{t} n^{t - r + 1} p^{\binom{t}{r} + \binom{t+1}{r} - 1}.
        \]
        Let us choose $c_1$ such that $c_3 n^{t - r + 1} p^{\binom{t+1}{r}-1} = \frac{1}{10\lambda}$. Recall that $\mu_D = \binom{k}{t} p^{\binom{t}{r}}$. We have
        \[
            \frac{10 \lambda c_3}{c_2} \ge \frac{\mu_D}{\mu_Y} \ge 10\lambda.
        \]

        Notice that if $X_S > 0.5\mu_D$ and $Z_S < \mu_D / (2\lambda)$, then $X_S > \lambda Z_S$ and thus $A_S$ holds. Therefore, the probability that $A_S$ does not hold is at most
        \begin{align*}
            \mathbb{P}(X_S \le 0.5 \mu_D) + \mathbb{P}(Z_S \ge \mu_D / (2\lambda)) &\le \mathbb{P}(X_S \le 0.5 \mu_D) + \mathbb{P}(Z_S \ge 5\mu_Y) \\ 
            &\le e^{-\mu_D/12} + e^{-5\mu_D(\log 5 - 1)c_2 / 10\lambda c_3} \le e^{-c_4 \mu_D},
        \end{align*}
        for some constant $c_4$ which does not depend on $c_0$ and $c_1$. The second inequality is true by \eqref{eq:X_S-bound} and \eqref{eq:Z_S-bound}. We have
        \begin{align*}
            \mu_D &\ge \left(\frac{k}{t}\right)^t p^{\binom{t}{r}} \ge \left(\frac{c_0}{t}\right)^t n^{\frac{t}{t-1} \cdot \frac{\binom{t}{r}(t+1-r)}{\binom{t+1}{r}-1}} \log^{\frac{t}{t-1}} n \cdot c_1^t n^{-\frac{\binom{t}{r}(t+1-r)}{\binom{t+1}{r}-1}} \\
            &= \left(\frac{c_0 c_1}{t}\right)^t n^{\frac{\binom{t}{r}(t+1-r)}{(\binom{t+1}{r}-1)(t-1)}} \log^{\frac{t}{t-1}} n.
        \end{align*}
        Hence,
        $$  
            \Pr(A_S) \le \exp\left(-c_4 \left(\frac{c_0 c_1}{t}\right)^t n^{\frac{\binom{t}{r}(t+1-r)}{\left(\binom{t+1}{r}-1\right)(t-1)}} \log^{\frac{t}{t-1}} n\right).
        $$
        
        By the union bound over all choices of $S$, the probability that there exists $S$ such that $A_S$ does not hold is at most
        \begin{align*}
            \Pr\left(\bigcup_S A_S \right) &\le \binom{n}{k} \exp\left(-c_4 \left(\frac{c_0 c_1}{t}\right)^t n^{\frac{\binom{t}{r}(t+1-r)}{\left(\binom{t+1}{r}-1\right)(t-1)}} \log^{\frac{t}{t-1}} n\right)\\
            &\le \exp\left(\left\lceil c_0 n^{\frac{\binom{t}{r}(t+1-r)}{\left(\binom{t+1}{r}-1\right)(t-1)}} \log^{\frac{t}{t-1}} n \right\rceil -c_4 \left(\frac{c_0 c_1}{t}\right)^t n^{\frac{\binom{t}{r}(t+1-r)}{\left(\binom{t+1}{r}-1\right)(t-1)}} \log^{\frac{t}{t-1}} n\right) \\
            &= o(1),
        \end{align*}
        where the last equality is true if $c_0$ and $c_1$ are sufficiently large.
    \end{proof}
    The proof of Lemma~\ref{lm:good_subhypergraph} is complete.
\end{proof}

\subsection{Proof of the upper bound of Theorem \ref{th: wsat}\ref{i: strong}}\label{sec: strong proof}

    Set $G \coloneqq G^r(n,p)$. Set $\alpha = \frac{1}{1 - p^{r-1}}$ and $\beta = \frac{1}{1 - p^{\binom{s}{r} - \binom{s - r}{r} - 1}}$. Unless stated otherwise, then the base of all the logs in the proof is $\alpha$. Set
    \[
        a_1 = \frac{1}{p} \log n \left(1 + \frac{3}{\log \log n}\right),\quad a_2=s \log_\beta (n^{r})\quad \text{and}\quad a_3 = \frac{a_2}{\left(\log a_2\right)^{1/r}}.
    \]
    
    Let $A_1, A_2, A_3$ be disjoint subsets of $V(G)$ of sizes $a_1, a_2, a_3$, respectively. Let $B$ denote the set of the remaining vertices of $V(G)$. Set $t = s - r$ and
    \[
        q = c_1 a_1^{-\frac{t + 1 - r}{\binom{t + 1}{r} - 1}}.
    \]
    If $t<r$, let $H_1$ be the empty hypergraph on $A_1$. Otherwise, if $t\ge r$, let $H_1\subseteq G[A_1]$ be the following spanning subhypergraph, which by Lemmas \ref{lemma:max-number-of-copies} and \ref{lm:good_subhypergraph} \textbf{whp} exists:
    \begin{itemize}
        \item $H_1$ is $K_{t+1}^r$-free.
        \item There exists a copy of $K_t^r$ in every induced subhypergraph of $H_1$ of size at least $$T = c_0 a_1^{\frac{\binom{t}{r}(t + 1 - r)}{\left(\binom{t+1}{r} - 1\right)(t-1)}} \log^{\frac{1}{t-1}} a_1.$$
        \item Every two vertices in $A_1$ are contained in at most $2 \binom{a_1}{t-2}q^{\binom{t}{r}}$ copies of $K_t^r$ in $H_1$.
    \end{itemize}
    Let $H$ be a subhypergraph of $G$ with the following edges:
    \begin{itemize}
        \item All edges of $H_1$.
        \item All edges of $G$ of the form $A'\cup B'$ where $A'\subseteq A_1$ with $1\le |A'|\le t$ and $B'\subseteq B$ with $0<|B'|=s-|A'|$.
    \end{itemize}
    We will say that an edge $e \in E(G) \setminus E(H)$ can be \textit{completed} if it closes a copy of $K_s^r$ together with $E(H)$. First, observe that there are no copies of $K_s^r$ in $H$. Indeed, consider a set $X$ of $s$ vertices from $A_1\cup B$, and suppose towards contradiction that $X$ induces a clique in $H$. Since $H[B]$ is an empty hypergraph, we have that $|X\cap B|\le r-1$, that is, $|X \cap A_1|\ge s-(r-1)=t+1$. Thus, if $t+1\ge r$, since $H[A_1]$ is $K_{t+1}^r$-free we have that $H$ is $K_s^r$-free. Otherwise, if $t+1\le r-1$, we note that there are no edges with $t+1$ vertices from $A_1$ and $r-(t+1)$ vertices from $B$.


    We will now show that most of the edges in $G[B]$ can be completed in $H$. Given an $(r-1)$-set of vertices $S\subset V$, the \textit{neighbourhood} of $S$ in $X\subset V$ is
    \begin{align*}
        N_{X}(S)\coloneqq \{v\in X\colon \{v\}\cup S\in E(G)\}.
    \end{align*}
    Furthermore, given an $(r-1)$-set of vertices $S \subset B$, we say that $S$ is \textit{good} if
    \[
        |N_{A_1}(S)| \ge m \coloneqq  \left(1 + \frac{2}{\log \log n}\right) \log n,
    \]
    and otherwise we say that $S$ is \textit{bad}. Moreover, we say that an edge $e \in G[B]$ is $\textit{good}$ if there exists at least one good $(r-1)$-subset $S \subset e$ and otherwise we say that $e$ is \textit{bad}. 

    \begin{claim}\label{appendix c: good edges}
        \textbf{Whp}, all good edges in $G[B]$ but $o(n^{r-1} \log n)$ can be completed.
    \end{claim}
    \begin{proof}
        Fix an $(r-1)$-set $S\subset B$ and expose all edges of the form $\{v\}\cup S$, $v\in A_1$. Let us assume that these edges justify the validity of the event that $S$ is good. Fix an $r$-subset $e\subset B$ containing $S$, noting that we then have that $e$ is good. We will show that the probability that $e$ cannot be completed is less than $\frac{2}{n}$. If so, then the expectation of the number of good edges in $G[B]$ which cannot be completed is $O(n^{r-1})$ and then, by Markov's inequality, there are \textbf{whp} $o(n^{r-1} \log n)$ such edges.

        Since $S$ is good, the size of the neighbourhood of $S$ in $A_1$, denoted by $N_{A_1}(S)$, is at least $m$. We will give an upper bound to the probability that there is no copy of $K_t^r \subseteq H[N_{A_1}(S)]$ which closes a copy of $K_s^r$ together with $e$. Notice that, for a given vertex $v \in N_{A_1}(S)$, the probability that $\{v\} \cup S' \in E(G)$ for all $S' \subset e$ of size $r-1$ is exactly $p^{r - 1}$. Define
        \[
            N_{S,e} \coloneqq \{v \in N_{A_1}(S) \colon \{v\} \cup S' \in E(G) \text{ for all } S' \subset e,\ |S'| = r-1\}. 
        \]
        Then, given $S$ and $N_{A_1}(S)$, $|N_{S,e}| \sim \text{Bin}(|N_{A_1}(S)|, p^{r-1})$. By Claim \ref{claim:binomial-bounds},
        \[
            \mathbb{P}\left(|N_{S,e}| \ge \frac{|N_{A_1}(S)|}{\log^2 |N_{A_1}(S)|}\right) \ge 1 - (1-p^{r-1})^{|N_{A_1}(S)| - \frac{|N_{A_1}(S)|}{\log |N_{A_1}(S)|}} \ge 1 - \frac{1}{n},
        \]
        where the last inequality holds since
        \[
            |N_{A_1}(S)| - \frac{|N_{A_1}(S)|}{\log |N_{A_1}(S)|} \ge m\left(1 - \frac{1}{\log \log n}\right) > \log n.
        \]

        Let us expose all edges of the form $\{v\}\cup S'$, $S'\subset e$, $v\in N_{A_1}(S)$, and let us suppose from now that $|N_{S,e}| \ge \frac{|N_{A_1}(S)|}{\log^2 |N_{A_1}(S)|}$ and thus $|N_{S,e}| > a_1 ^ {1-\epsilon}$ for some small constant $\epsilon > 0$. Denote the family of copies of $K_t^r$ in $H[N_{S,e}]$ by $\mathcal{K}$. Since in every induced subhypergraph of $H[N_{S,e}]$ of size $T$ there is a copy of $K_t^r$ and every copy of $K_t^r$ appears in at most $\binom{|N_{S,e}| - t}{T - t}$ subsets of size $T$, we have
        \[
            |\mathcal{K}| \ge \frac{\binom{|N_{S,e}|}{T}}{\binom{|N_{S,e}| - t}{T - t}} > \left(\frac{|N_{S,e}|}{T}\right)^{t} = \Omega\left(a_1^{t - t\epsilon - \frac{t \binom{t}{r} (t + 1 - r)}{\left(\binom{t+1}{r} - 1\right)(t-1)}} \log^{-\frac{t}{t-1}} a_1\right).
        \] 
        
        By Claim \ref{claim:t-condition1_ap}, 
        \[
            \left(\frac{|N_{S,e}|}{T}\right)^{t} = \Omega\left(a_1^{r - 1 - t\epsilon}\log^{-\frac{t}{t-1}} a_1\right) \ge a_1^{1 + \epsilon_1},
        \]
        where the last inequality is true for some small enough $\epsilon_1 > 0$.

        Denote by $F$ the number of pairs $(K, K') \in \mathcal{K}$ satisfying $|V(K) \cap V(K')| \ge 2$. We have that $F = O\left(|\mathcal{K}| a^{t-2} q^{\binom{t}{r}}\right)$. Note that $\frac{1}{T^t} = \Omega\left(\frac{q^{\binom{t}{r}} \log a_1}{T}\right)$ and thus 
        $$
            |\mathcal{K}| = \Omega\left(a_1^{t - t \epsilon} q^{\binom{t}{r}} T^{-1} \log a_1\right) \ge a_1^{t - 2t \epsilon} q^{\binom{t}{r}} T^{-1}.
        $$
        Then,
        \begin{align*}
            \frac{|\mathcal{K}|^2}{F} = \Omega\left(\frac{|\mathcal{K}|}{a_1^{t-2} q^{\binom{t}{r}}}\right) = \Omega\left(\frac{a_1^{t - 2t \epsilon}}{Ta_1^{t-2}}\right).
        \end{align*}
        By Claim \ref{claim:t-condition1_ap}, $T \le a_1^{1-\frac{1}{t}}$. Hence,
        \begin{align}\label{align:Kappa-F-ratio}
            \frac{|\mathcal{K}|^2}{F} = \Omega\left(a_1^{2 - 2t\epsilon - 1 + \frac{1}{t}}\right) \ge a_1^{1 + \epsilon_2},
        \end{align}
        for some $\epsilon_2 > 0$.

        Recall that we want to give an upper bound for the probability that $e$ cannot be completed. For every $K \in \mathcal{K}$, denote by $B_K$ the event that $G[K \cup e] \cong K_s$. Note that, if $B_K$ holds, then $e$ can be completed. We will give an upper bound for  $\Pr\left(\cap_{K \in \mathcal{K}} B_K^c\right)$ using Theorem \ref{theorem:janson}. For every $K \in \mathcal{K}$, we already have the edges induced by $K$ and edges with one vertex from $K$ and $r-1$ vertices from $e$. Therefore, two events $B_K$ and $B_{K'}$ are dependent only if $|K \cap K'| \ge 2$. Moreover, the probability that $B_K$ holds is at most $p^{\binom{s}{r}} = O(1)$. We have,
        \begin{align*}
            \mu &\coloneqq \sum_{K \in \mathcal{K}} \Pr(B_k) = \Theta(|\mathcal{K}|), \\
            \Delta &\coloneqq \sum_{\substack{K, K' \in \mathcal{K} \\ |V(K) \cap V(K')| \ge 2}} \Pr(B_K \cap B_{K'}) < F.
        \end{align*}
        Hence, by Theorem \ref{theorem:janson}, 
        \begin{align*}
            \Pr\left(\cap_{K \in \mathcal{K}} B_K^c\right) &\le e^{-\mu^2 / 2\Delta} \le e^{-\Theta(|\mathcal{K}|^2 / F)} \le e^{-a_1^{1 + \epsilon_3}} < \frac{1}{n},
        \end{align*}
        where the third inequality is true for some $0<\epsilon_3<\epsilon_2$ by \eqref{align:Kappa-F-ratio}. Therefore, the expected number of uncompleted good edges in $G[B]$ is $O(n^{r-1})$ and thus, by Markov's inequality, \textbf{whp} there are $o(n^{r - 1} \log n)$ uncompleted good edges in $G[B]$.
    \end{proof}

    Before dealing with bad edges in $G[B]$, the following claim gives an upper bound to the number of bad $(r-1)$-sets in $B$.
    \begin{claim}\label{claim:bad-sets-in-B}
        There are at most $\frac{n^{r-1}}{\log n}$ bad $(r-1)$-subsets $S \subset B$.
    \end{claim}
    \begin{proof}
        Fix an $(r-1)$-subset $S \subset B$. Recall that
        \[
            N_{A_1}(S) = \{v \in A_1 \colon \{v\} \cup S \in E(G)\}.
        \]
        Note that $|N_{A_1}(S)| \sim \text{Bin}(a_1, p)$. By Claim \ref{claim:binomial-bounds},
        \[
            \mathbb{P}(|N_{A_1}(S)| < m) = \mathbb{P}\left(|N_{A_1}(S)| < a_1 p - \frac{\log n}{\log \log n}\right) \le e^{-\log n / 4 \log \log n} = o(1/\log n).
        \]
        Hence, the claim is true by Markov's inequality.
    \end{proof}

    We will use $A_2$ to complete bad edges in $G[B]$. By Lemma \ref{lm:good_subhypergraph}, there exists a spanning subhypergraph $H_2 \subseteq G[A_2]$ which is $K_{t+1}^r$-free and, for some sufficiently small $\delta$ with respect to $t$, every subset of $a_2^{1-\delta}$ vertices induces a copy of $K_t^r$. Observe that there are at least
    \[
        \frac{a_2}{t} - a_2^{1-\delta} \ge \frac{a_2}{s} = \log_\beta (n^r)
    \]
    vertex-disjoint copies of $K_t^r$ in $H_2[A_2]$.

    Add to $H$ the following edges:
    \begin{itemize}
        \item edges from $H_2$,
        \item edges intersecting both $A_2$ and $B$ in the graph $G$ with at least $2$ and at most $t$ vertices from $A_2$,
        \item edges of the form $S\cup \{v\}$ where $S\subseteq B$ is a bad $(r-1)$-subset and $v\in A_2$.
    \end{itemize}
    Before showing how to complete bad edges utilising these edges, let us first note that there are still no copies of $K_s^r$ in $H$. As we only added edges induced by $A_2\cup B$, and there are no edges intersecting both $A_1$ and $A_2$, we may consider a set $X\subseteq A_2\cup B$ of $s$ vertices, and suppose towards contradiction that $X$ induces a clique in $H$. Since $H[B]$ is empty, we have that $|X\cap B|\le r-1$ and $|X\cap A_2|\ge s-r+1=t+1$. Similar to before, if $t+1\ge r$, we have a contradiction since $H[A_2]$ is $K_{t+1}^r$-free, and, if $t+1\le r-1$, then note that we have not added to $H$ edges with $s-r+1=t+1$ vertices from $A_2$ and the rest from $B$, and we once again obtain a contradiction.

    Let us now consider some bad edge $e$ in $G[B]$. Given a copy of $K_t^r \subseteq H[A_2]$, the probability that this copy fails to complete $e$ is exactly $1/\beta$. Since we have at least $\log_\beta(n^r)$ vertex-disjoint copies of $K_t^r$ in $H[A_2]$, the probability that $e$ cannot be completed is at most $n^{-r}$. By Claim \ref{claim:bad-sets-in-B}, the number of bad edges in $B$ is at most $\frac{n^{r-1}}{\log n} \cdot n$. Therefore, the expectation of the number of bad edges that cannot be completed is at most $1 / \log n$. By Markov's inequality, \textbf{whp} all the bad edges in $G[B]$ can be completed.

    Next, we consider edges of the form $S \cup \{v\}$ where $S$ is a good $(r-1)$-set in $B$ and $v \in A_2$. We have $\Theta(n^{r-1} \log n)$ edges of this form, so we are not allowed to add them to $H$; thus, we will have to complete most of them. We will use $A_3$ for that.

    Once again, \textbf{whp} there exists a spanning subhypergraph $H_3 \subset G[A_{3}]$ which is $K_{t+1}^r$-free and such that every set of $a_3^{1-\delta}$ vertices in $H_3$, for some small constant $\delta>0$, contains a copy of $K_t^r$. Let us add to $H$ the following edges:
    \begin{itemize}
        \item edges from $H_3$,
        \item edges intersecting both $B$ and $A_3$ in the graph $G$ with at least two and at most $t$ vertices from $A_3$,
        \item edges of the form $S\cup \{v\}$ where $S\subseteq B$ is a good $(r-1)$-subset and $v\in A_{3}$,
        \item edges intersecting both $B\cup A_{2}$ and $A_{3}$ in the graph $G$ with exactly one vertex from $A_{2}$ and at most $t$ vertices from $A_{3}$.
    \end{itemize}
    Let us first show that $H$ is still $K_s^r$-free. Since no edges are intersecting both $A_1$ and $A_2\cup A_3$, it suffices to consider $X\subseteq B\cup A_{2}\cup A_{3}$ of size $s$, and suppose towards contradiction that $X$ induces a clique in $H$. Once again, we note that $|X\cap B|\le r-1$. Now, if there are no vertices from $A_3$, we are done by the previous arguments, as we did not add new edges which do not intersect $A_3$. If there are at least $t+1$ vertices from $A_{3}$, we are done by similar arguments to before. Furthermore, if there are at least two vertices in $A_2$, note that there are no edges in $H$ containing two vertices from $A_2$ and at least one vertex from $A_3$, leading to a contradiction. Thus, we are left with the case where $|X\cap A_{2}|=1, |X\cap A_{3}|=t, |X\cap B|=r-1$. Now, if $X\cap B$ is bad, there are no edges of the form $(X\cap B)\cup \{v\}$, $v\in A_3$, in $H$, and if it is good, then there are no edges of the form $(X\cap B)\cup \{v\}$, $v\in A_2$, in $H$ --- either way, we get that $X$ cannot induce a clique.
    
    Let us now consider an edge $e \in G$ of the form $S \cup \{v\}$ where $S$ is a good $(r-1)$-set in $B$ and $v \in A_{2}$. As before, there are at least $a_3/s$ vertex-disjoint copies of $K_t^r$ in $H[A_{3}]$. The probability that a given copy fails to complete $e$ is exactly $1/\beta$. Thus, the probability that $e$ cannot be completed is at most $\beta^{-a_3 / s}$. Therefore, the expectation of such edges $e$ that cannot be completed is at most
    \[
        n^{r-1} a_2 \beta^{-a_3 / s} \ll n^{r-1} \log n,
    \]
    and thus, by Markov's inequality, \textbf{whp} all but at most $o(n^{r-1} \log n)$ edges of the form $S \cup \{v\}$, where $S$ is a good $(r-1)$-set in $B$ and $v \in A_2$, are completed.\\

    Let us count the edges in $H$:
    \begin{itemize}
        \item Edges intersecting both $B$ and $A_1$. There are $(1+o(1)) \binom{n}{r-1} \log n$ such edges.
        \item Edges intersecting both $B$ and $A_2$. There are $o(n^{r-1})$ such edges.
        \item Edges intersecting both $B\cup A_{2}$ and $A_{3}$. There are $o(n^{r-1} \log n)$ such edges.
        \item Some of the edges induced by $A_1\cup A_2\cup A_3$, of which there are at most $o(n)$.
    \end{itemize}
    Note that we have completed all edges except for the following: we did not complete edges of the form $S\cup \{v\}$ where $S\subseteq B$ is bad and $v\in A_3$ as well as some of the edges induced by $A_1\cup A_2\cup A_3$. Furthermore, if $t<r$, we have not completed edges intersecting both $B$ and $A_1\cup A_2\cup A_3$ with more than $t$ vertices in either $A_1$, $A_2$, or $A_3$. However, there are at most $o(n^{r-1}\log n)$ edges of these types, and we are thus done by Observation \ref{obs: small amount of edges}.
    Hence, $H$ has $(1+o(1)) \binom{n}{r-1} \log n$ edges and $o(n^{r-1} \log n)$ additional edges make it saturated.

\end{document}